\documentclass[12pt,oneside]{article}
\newcommand{\Path}{}
\newcommand{\figs}{}
\usepackage{ \Path Paper_style}
\usepackage{ \Path Keywords}
\usepackage{ \Path Environments}
\usepackage{ \Path Research_keywords}
\usepackage{tikz-cd}
\usepackage{tabularx} \newcolumntype{L}{>{\raggedright\arraybackslash}X}

\newcommand{\reconstruct}{\mathcal{T}}

\newcommand{\cocyc}{\mathcal{G}}
\newcommand{\GrowthT}{\Psi}
\newcommand{\Hbasis}{\mathfrak{w}}
\newcommand{\W}{\mathfrak{S}}
\DeclareMathOperator{\dev}{dev}
\newcommand{\mult}{ C_{\text{NUH}} }
\newcommand{\Csens}{ C_{\phi,\Phi} }
\DeclareMathOperator{\retract}{ret}
\newcommand{\Csensr}{ C_{\phi,\Phi, \retract} }
\DeclareMathOperator{\AutCor}{AutCorr}
\title{Learning theory for dynamical systems}
\author{Tyrus Berry, Suddhasattwa Das  \thanks{Department of Mathematical Sciences, George Mason University} \thanks{The research was funded by the NSF Sponsored Program fund \textbf{204839}.}}
\begin{document}
\maketitle
\begin{abstract} The task of modelling and forecasting a dynamical system is one of the oldest problems, and it remains challenging. Broadly, this task has two subtasks - extracting the full dynamical information from a partial observation; and then explicitly learning the dynamics from this information. We present a mathematical framework in which the dynamical information is represented in the form of an embedding. The framework combines the two subtasks using the language of spaces, maps, and commutations. The framework also unifies two of the most common learning paradigms - delay-coordinates and reservoir computing. We use this framework as a platform for two other investigations of the reconstructed system - its dynamical stability; and the growth of error under iterations. We show that these questions are deeply tied to more fundamental properties of the underlying system - the behavior of matrix cocycles over the base dynamics, its non-uniform hyperbolic behavior, and its decay of correlations. Thus, our framework bridges the gap between universally observed behavior of dynamics modelling; and the spectral, differential and ergodic properties intrinsic to the dynamics.
\end{abstract}



%-_-_-_-_-_-_-_-_-_-_-_-_-_-_-_-_-_-_-_-_-_-_-_-_-_-_-_-_-_-_-_-_-_-_-_-_-_-_-_-_-_-_-_-_-_-_-_-_-_-_-_-_-_-_-_-_-_-_-_-_-_-_-_-_-_-_-_-_-_-_-_-_-_-_-_-
\section{Introduction} \label{sec:intro}

Many investigations of physical systems involve modeling and forecasting a dynamical system, in fields as diverse as climate sciences \cite{SlawinskaGiannakis2017}, traffic dynamics \cite{DasEtAl_Alfaya_ACC_2021, RahmanHasan2020} or epidemiology \cite{MustaveeEtAl_covid_2021}. With the growth of computational power, many new techniques and paradigms of reconstructing a dynamical system have been proposed, we call this the \emph{learning problem} for dynamics. Most of the common techniques seek to recreate a dynamical system by developing conjugate or equivalent dynamical system, usually in a higher dimensional space. We present a theoretical framework which unifies these techniques. The framework, presented in the form of a commuting diagram \eqref{eqn:paradigm} of maps and operators, helps to identify and distinguish between different conceptual components of these learning. %We use the language of Koopman operators instead of state-space equations.

Figure~\ref{fig:outline} presents an outline of the paper. The primary requirement of all learning techniques is an embedding of the dynamics (Assumption~\ref{A:pPhi}), which may be explicit or implicit. We show in Section~\ref{sec:paradigm} that there are two main paradigms of learning dynamics - based on invariant graphs and delay coordinates respectively, in which the embedding is implicit and explicit respectively. We unify both these paradigms in a common abstract mathematical framework \eqref{eqn:paradigm} and show how the unknown dynamics can be reconstructed as a conjugate dynamical system \eqref{eqn:feedback_commut} in the embedding space. We next investigate the stability of these reconstructions. We show that a proper quantitative assessment of the stability is related to the original dynamics, as well as the kind of interpolation done by the learning technique.  

\begin{figure}\center
\includegraphics[width=.95\linewidth]{\figs 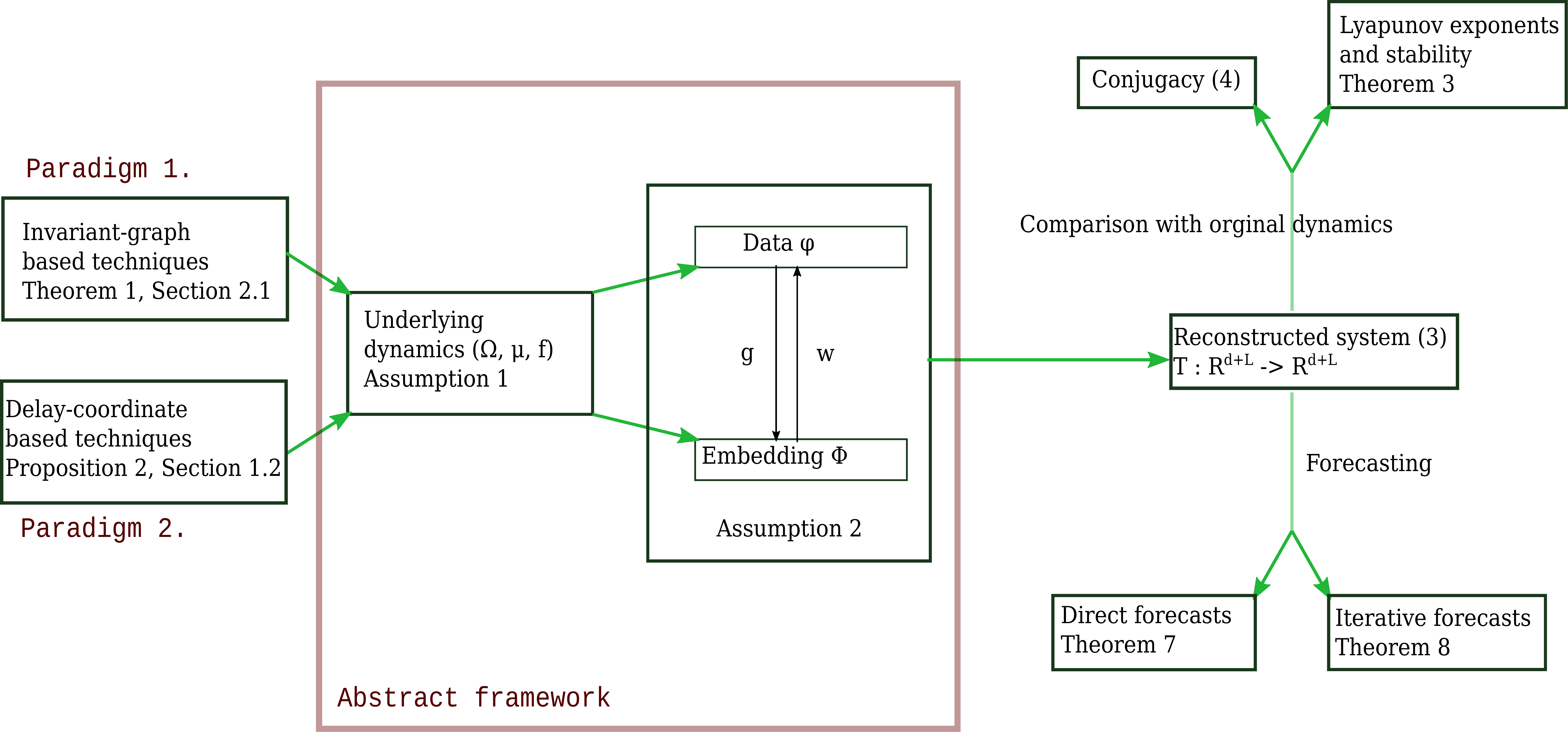}
\caption{ Outline of results and theory.}
\label{fig:outline}
\end{figure}

Another important consideration for us is the effectiveness of reconstruction models, for the purpose of forecasting. The forecasting can be of two types, \emph{direct} or \emph{iterative}. We show in Theorem~\ref{thm:direct} that the direct method is limited by the rate of decay of correlations of the  system. On the other hand, the iterative method is deeply connected to the embedding properties of the data, as well as the learning scheme employed. A key aspect of learning theory is the choice of a hypothesis space. This functional analytic consideration also fits seamlessly with our framework. We show how the rate at which the learned dynamics and the true dynamics diverge, is a combination of the intrinsic dynamical properties as well as the effectiveness of the hypothesis space. We do so using the language of matrix cocycles. See Figure~\ref{fig:compare} for a comparative illustration of two computation techniques. We do an extensive comparison of various learning techniques in tables \ref{tab:prdgm}, \ref{tab:learn}. %, \ref{tab:techniques}.

\begin{figure}\center
\includegraphics[width=.48\linewidth]{\figs 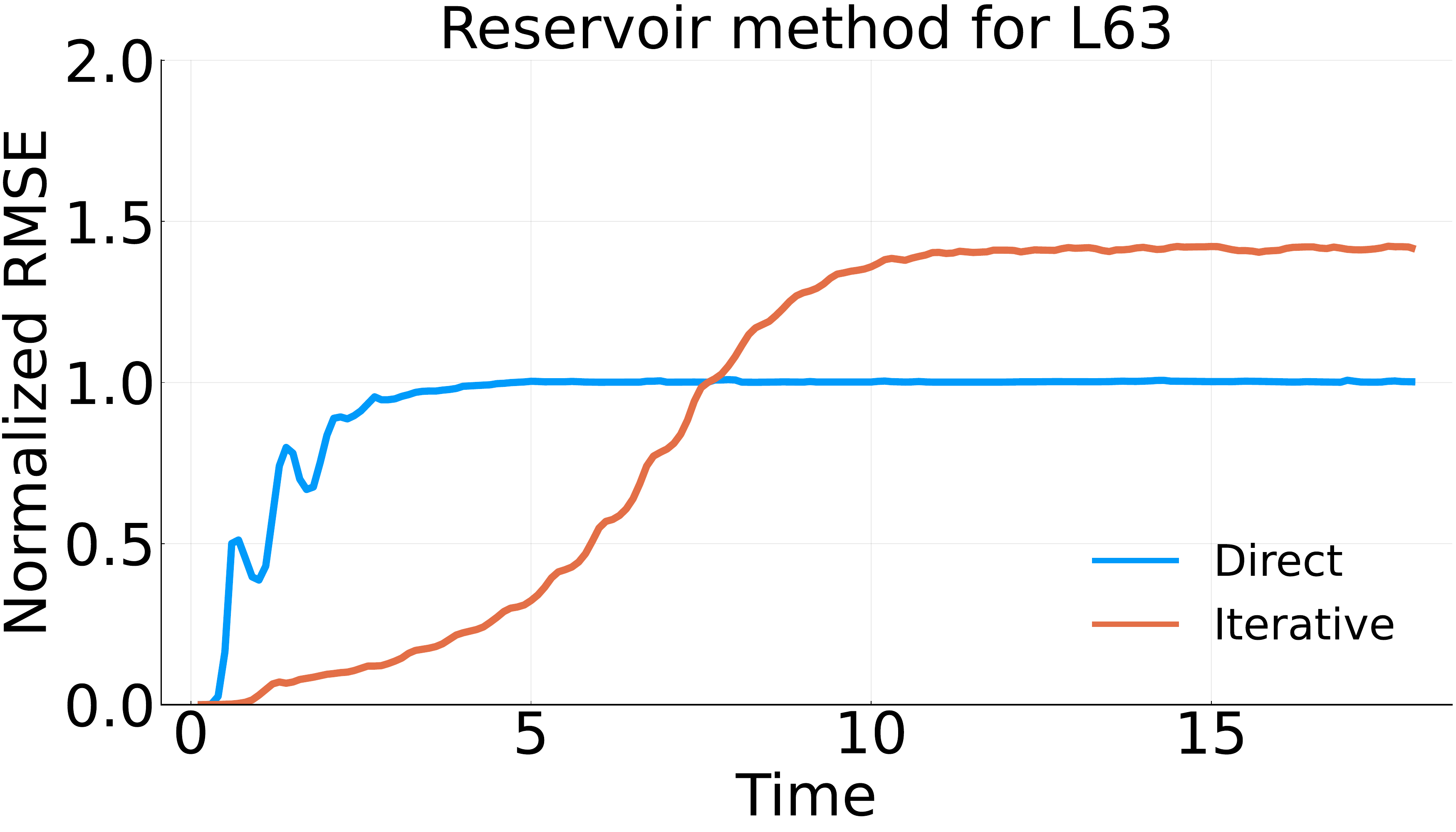}
\includegraphics[width=.48\linewidth]{\figs 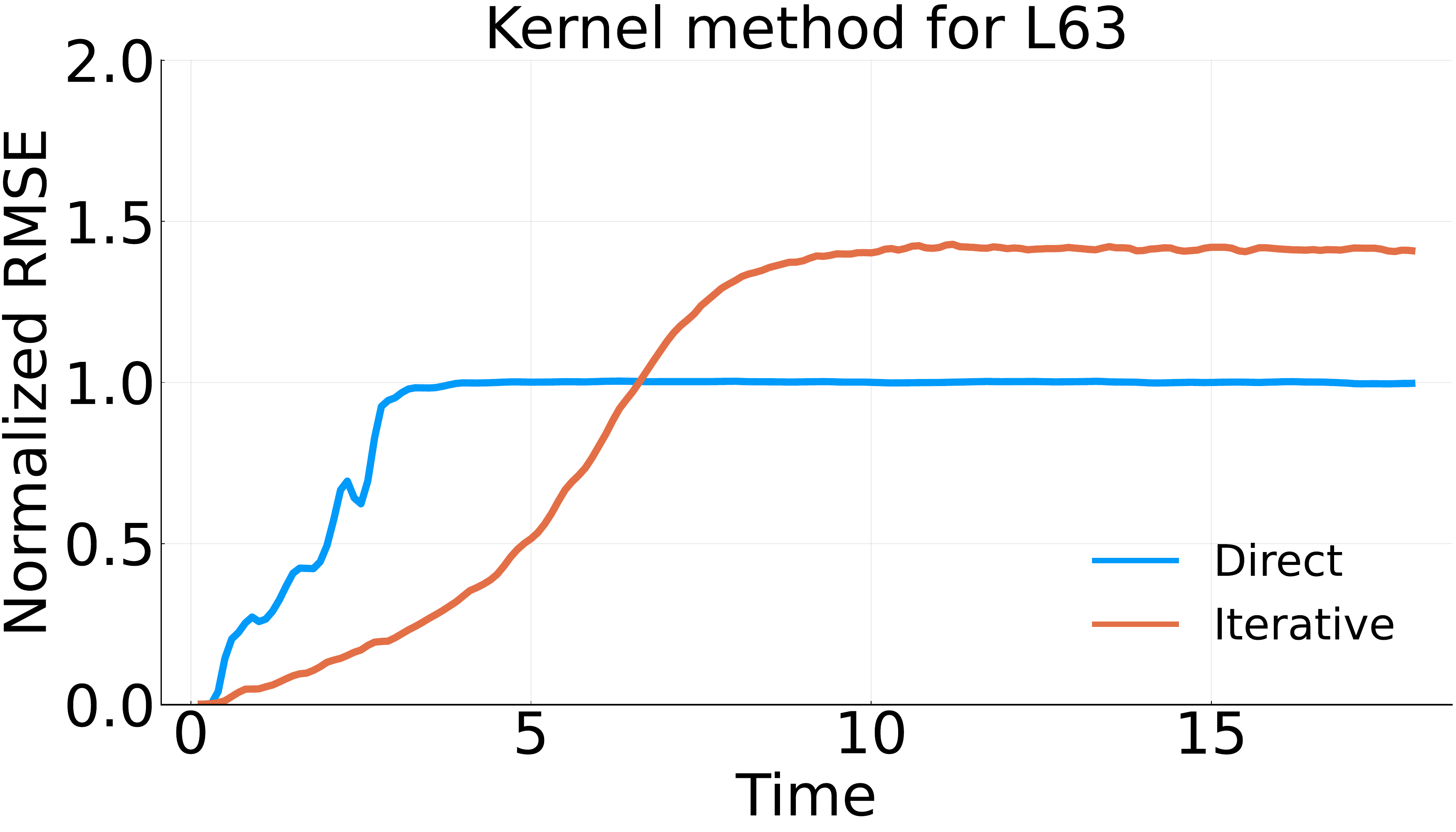}
\caption{ Results of forecasting attempts on the Loren63 system, using the invariant-graph paradigm( Section \ref{sec:inv_graph}) on the left, and the delay-coordinate based paradigm (Section \ref{sec:reg}) on the right. For both these paradigms, we compare the performance of direct and iterative modes of forecasting (see \eqref{eqn:def:err_iter}, \eqref{eqn:def:err_direct} for definitions). The horizontal axis shows the forecast-time $n$, and the vertical axis is the RMS error of forecast as a function of $n$, for a signal of unit $L^2(\mu)$-norm. The RMS error is meant to approximate the $L^2(\mu)$ norm of the error as a function of the initial state of the underlying system. The iterative errors are seen to increase and eventually settle around $\sqrt{2}$, while the error from the direct mode settles at $q$. We show that this is a universal behavior, based on a mathematical framework \eqref{eqn:paradigm} that unifies both these paradigms, and both these modes. Based on this framework, we develop Theorems \ref{thm:direct} and \ref{thm:iterative} which provide expressions for the asymptotic behavior of these errors, and are consistent with these graphs. Also see Section~\ref{sec:discuss} for an extended analysis. %on the expected performances of these methods for mixing systems such as L63.
}
\label{fig:compare}
\end{figure}

\paragraph{The framework} We now build our general abstract framework by assigning mathematical objects and assumptions to various components of the entire prediction scheme. We begin with the dynamical system itself. A common practice is to assume an unknown dynamical system satisfying

\begin{Assumption} \label{A:f}
There is a $C^1$ dynamical system $f: \tilde\Omega \to \tilde\Omega$ on an $m$-dimensional $C^1$ manifold $\tilde\Omega$, with an ergodic measure $\mu$ with compact support $\Omega$.
\end{Assumption}

This minimal assumption on the underlying system allows it to be applied in a many situations. The assumption of an ergodic invariant measure with compact support is fulfilled in any system with bounded trajectories. This assumption is however only for theoretical purposes, as the system in Assumption~\ref{A:f} is usually not presented explicitly. We next provide an abstract framework which describes how the system is manifested in a data-driven setting.

\begin{Assumption} \label{A:pPhi}
There are maps $\phi:\Omega \to \real^d$ and $\Phi : \Omega\to \real^L$ such that 
\begin{enumerate} [(i)]
    \item $\Phi$ is an injective map;
    \item there is a function $g:\real^d\times \real^L \to \real^L$ such that $\Phi\circ f = g\circ \left(\phi \times \Phi\right)$.
\end{enumerate}
\end{Assumption} 

The map $\phi$ is the measurement through which the dynamical system is observed.  So the codomain of $\phi$ is often low dimensional and may be only partially observe the system. Since $\Phi$ is an embedding, it effectively serves as a representation of the dynamics-space $\Omega$ in $\real^L$ space.  The function $g$ will be known and computable explicitly. Note that $\Phi\circ f$ is the evolution of $\Phi$ under one iteration of the dynamics of $f$. Thus $g$ contains and encodes the evolution law, in terms of the current states of $\Phi$ and $\phi$. This is summarized in the left half of the commutative diagram in \eqref{eqn:paradigm}. The function $g$ is usually known explicitly, while $\Phi$ could be explicit (such as in delay-coordinate techniques Section~\ref{sec:reg}); or implicit (such as in invariant-graph based techniques, Section~\ref{sec:inv_graph}). Let $\dim_\mu$ denote the box-dimension of the invariant set supported by the measure $\mu$. We typically have
\[ d \leq \dim_\mu << L . \]
The task now is to reconstruct the dynamics using the maps $\phi, \Phi$. This reconstruction will not be via the phase space $\tilde{\Omega}$ of the dynamics but rather through various function spaces which try to capture how these maps are transformed under the dynamics. For this reason, we shall use the operator theoretic language of dynamical systems. This is done using a operator-theoretic representation of the dynamics known as the \emph{Koopman operator}.

\paragraph{Koopman operator} The Koopman operator \cite{Halmos1956, DasGiannakis_delay_2019, DasGiannakis_RKHS_2018} is a time-shift operator, it acts on observables by composition with the map $f$. The space $L^2(\Omega,\mu)$ of square-integrable, complex-valued functions will be called our \emph{space of observables}. The space $L^2(\Omega,\mu)$ can be unambiguously abbreviated as $L^2(\mu)$. For every $n\in\num$, the operator $U^n : L^2(\mu) \to L^2(\mu)$ is defined on every $h\in L^2(\mu)$ as
\[ (U^n h ) : \omega \mapsto (h \circ f^n)(\omega), \mbox{ for } \mu-a.e. \omega \in \Omega. \]
The operator $U$ and all its iterates $U^n$ are unitary maps. The constant function $1_\Omega$ is always an eigenfunction for $U$, with eigenvalue $1$. A consequence of $\mu$ being ergodic is that $1$ is a simple eigenvalue. Let $\Disc$ denote the closure in $L^2(\mu)$ of all eigenfunctions of $U$. Then one has the orthogonal decomposition
\begin{equation} \label{eqn:def:L2split}
L^2(\mu) = \Disc \oplus \Disc^\bot .
\end{equation}
The system is said to be \emph{mixing} if the space $\Disc$ consists of only constant functions. Koopman operators allow the study of arbitrary nonlinear dynamics as linear dynamics on infinite dimensional vector spaces. It has not only been used for forecasting tasks \cite[e.g.][]{MustaveeEtAl_covid_2021}, but also in tasks such as harmonic analysis of dynamics generated data \cite[e.g.][]{DasGiannakis_RKHS_2018}, control \cite{korda2020optimal}, and detection of coherent patterns \cite[e.g.][]{DasGiannakis_delay_2019, GiannakisEtAl2015}.

We are now prepared to present the reconstructed dynamics in terms of the Koopman operator.

\paragraph{Feedback function} Since $\Phi$ is an embedding (by Assumption \ref{A:pPhi}), the current state of $\Phi$ determines the current and all future states in $\Omega$ and therefore of $\phi$. Therefore for every $k\in\num$, there is a function $w_k$ such that
\begin{equation} \label{eqn:def:wk}
w_k:\real^L \to \real^d, \quad w_k \circ \Phi = U^k \phi = \phi \circ f^k. 
\end{equation}
The learning task is to learn this map $w_k$. The following diagram connects all the maps and spaces we have discussed so far.
\begin{equation} \label{eqn:paradigm}
\begin{tikzcd}[row sep = large, column sep = large]
\ &\real^d &\real^d\times \real^L \arrow[bend right = 90]{dll}{\proj_1} \arrow{r}{\proj_2} \arrow{l}{g} & \real^L \arrow{r}{w_k} &\real^d \\
\real^d &\Omega \arrow{u}{\Phi} &\Omega \arrow{u}{\phi\times\Phi} \arrow{ur}{\Phi} \arrow{rr}[swap]{f^k} \arrow[bend right = 15]{urr}{U^k \phi} \arrow{ul}{U\Phi} \arrow{l}[swap]{f} \arrow[bend left = 30]{ll}[swap]{\phi} &\ &\Omega \arrow{u}[swap]{\phi} 
\end{tikzcd}\end{equation}
We next look at the specific case when $k=1$.

\paragraph{The reconstructed system} When $k=1$, $w_k$ will be denoted as $w$. The following standalone or reconstructed dynamical system on $\real^d \times \real^L$ :
\begin{equation} \label{eqn:def:feedback_1}
\reconstruct : \real^d \times \real^L \to \real^d \times \real^L, \quad 
\left[ \begin{array}{c} u_{n+1} \\ y_{n+1} \end{array} \right]
 = \reconstruct \left[ \begin{array}{c} u_{n} \\ y_{n} \end{array} \right] 
  = \left[ \begin{array}{c} w \left( y_n \right) \\ g\left( u_{n}, y_n \right) \end{array} \right].
\end{equation}
is conjugate to the dynamics on the attractor in $\Omega$. We explore this conjugation further using the commutation diagram in \eqref{eqn:feedback_commut}. The blue loop reveals the conjugacy, the green loops are the result of Assumption~\ref{A:pPhi}, and the red loop is the result of \eqref{eqn:def:wk}.
\begin{equation}\label{eqn:feedback_commut}
\begin{tikzcd}[row sep = large, column sep = huge]
\ &\  & \green{\real^L} \\
\ & \blue{ \real^d\times\real^L } \arrow[green, bend left = 30]{ru}{g} \arrow[blue, dashed]{r}[swap]{w\times g} \arrow{dl}[swap]{\proj_2} & \blue{ \real^d\times\real^L} \arrow[green]{u}{\proj_2} \arrow{ddl}{\proj_1}\\
\itranga{\real^L} \arrow[Itranga]{dr}[swap]{w} & \blue{\Omega} \arrow[green, bend left = 10]{uur}{U\Phi} \arrow[blue, dashed]{r}{f} \arrow[blue, dashed]{u}{\phi\times\Phi} \arrow[Itranga]{l}{\Phi} \arrow[Itranga]{d}[swap]{U\phi} \arrow{ur}{U\phi \times U\Phi} & \blue{\Omega} \arrow[green, bend right = 50]{uu}[swap]{\Phi} \arrow{dl}{\phi} \arrow[blue, dashed]{u}[swap]{\phi\times\Phi} \\
\ & \itranga{\real^d} &\ 
\end{tikzcd}
\end{equation}
The system \eqref{eqn:def:feedback_1} evolves in parallel to $f:\Omega\to\Omega$, while maintaining a conjugacy via the map $\phi\times\Phi$. The state vector $z_n := \left( u_n, y_n \right)$ is interpreted as follows : $u_n$ is represents the value $\phi(\omega_n)$, which is the unknown state $\omega_n$ observed through $\phi$. $y_n\in \real^L$ is the embedded point $\Phi(\omega_n)$ in $\real^L$. With this interpretation in mind, \eqref{eqn:def:feedback_1} is to be initialized with the state vector
\begin{equation} \label{eqn:feedback_1_init}
z_0 = z_0(\omega_0) := \left( \phi(\omega_0), \Phi(\omega_0) \right) \in \real^d \times \real^L .
\end{equation}
for some $\omega_0 \in \Omega$. The only component of \eqref{eqn:feedback_commut} that needs to be learned is the feedback $w$. One of our main focus is the situation when $w$ cannot be determined exactly but can only approximated. In that case, one needs to consider a hypothesis space, and the error built into the model during learning.

\paragraph{Hypothesis space} In a practical situation, $w_k$ is estimated from a search-set or hypothesis space $\mathcal{H}$ \citep[e.g.][]{AlxndrGian2020}, which may be a linear subspace or non-linear collection of functions. Thus the true function $w_k$ can be expressed as
\begin{equation} \label{eqn:def:hat_w}
w_k = \hat{w}_k + \Delta w_k ,
\end{equation}
where $\hat{w}_k$ is the estimated function, and $\Delta w_k$ is the error. Thus \eqref{eqn:paradigm} can be rewritten as :
\[\begin{tikzcd}[row sep = large, column sep = large]
\ &\real^d &\real^d\times \real^L \arrow[bend right = 90]{dll}{\proj_1} \arrow{r}{\proj_2} \arrow{l}{g} & \real^L \arrow{d}[swap]{w_k} \arrow[bend right=20]{ddr}{\hat{w}_k\times \Delta w_k} \arrow[bend left = 40]{r}[swap]{\hat{w}_k} \arrow[bend left = 30]{dr}[swap]{\Delta w_k} &\real^d \\
\real^d &\Omega \arrow{u}{\Phi} &\Omega \arrow{u}{\phi\times\Phi} \arrow{ur}{\Phi} \arrow{dr}[swap]{f^k} \arrow{r}[swap]{U^k \phi} \arrow{l}[swap]{f} \arrow[bend left = 30]{ll}[swap]{\phi} &\real^d &\real^d \\
\ &\ &\ &\Omega \arrow{u}[swap]{\phi} &\real^d\times\real^d \arrow[bend left = 15]{ul}{+} \arrow{u}[swap]{\proj_2} \arrow[bend right = 60]{uu}[swap]{\proj_1}
\end{tikzcd}\]
Similarly to \eqref{eqn:def:feedback_1}, the dynamics under the approximated feedback function becomes
\begin{equation} \label{eqn:def:feedback_2}
\hat{\reconstruct} : \real^d \times \real^L \to \real^d \times \real^L, \quad 
\left[ \begin{array}{c} u_{n+1} \\ y_{n+1} \end{array} \right]
= \hat{\reconstruct} \left[ \begin{array}{c} u_{n} \\ y_{n} \end{array} \right] 
= \left[ \begin{array}{c} \hat{w} \left( y_n \right) \\ g\left( u_{n}, y_n \right) \end{array} \right].
\end{equation}
The difference between \eqref{eqn:def:feedback_1} and \eqref{eqn:def:feedback_2} is the basis for evaluating the error of forecasts. We show in Section~\ref{sec:forecast} that the asymptotic rate at which these two systems diverge could depend on both the spectral properties of the dynamics, or its Lyapunov exponents. This completes the description of our framework.

\paragraph{Outline} The rest of the paper is organized as follow. In Section~\ref{sec:paradigm} we explore how the two main paradigms of learning dynamical systems fall under our framework. In Section~\ref{sec:stability} we use \ref{eqn:paradigm} to derive stability results for the reconstructed system. In Section~\ref{sec:forecast}, we again use \ref{eqn:paradigm} to obtain the rate of growth of errors when performing forecasts with the reconstructed system. We illustrate our results with some simple numerical examples in Section~\ref{sec:discuss}. Section~\ref{sec:cocyc}, \ref{sec:cocyc:peturb} Sections \ref{sec:proof:kncd3l}, \ref{sec:proof:lambda1}, \ref{sec:proof:direct} and \ref{sec:proof:iterative} contain the proofs of our theorems.   

%-_-_-_-_-_-_-_-_-_-_-_-_-_-_-_-_-_-_-_-_-_-_-_-_-_-_-_-_-_-_-_-_-_-_-_-_-_-_-_-_-_-_-_-_-_-_-_-_-_-_-_-_-_-_-_-_-_-_-_-_-_-_-_-_-_-_-_-_-_-_-_-_-_-_-_-
\section{Two paradigms of learning} \label{sec:paradigm} 

We now examine the two most important paradigms for realizing the scheme in \eqref{eqn:paradigm} - delay-coordinates and invariant-graph based techniques. We show that they follow the framework described in Assumptions~\ref{A:f} and \ref{A:pPhi}, and how their implementations are special cases of \eqref{eqn:def:hat_w}, \eqref{eqn:def:feedback_1} and \eqref{eqn:feedback_1_init}. The Table~\ref{tab:prdgm} summarizes some features of these two techniques. Note that the questions of producing, constructing, or computing $\Phi$ and $g$ are separate from the question of constructing $w_k$ and $\hat{w}_k$. Table~\ref{tab:learn} gives an overview of various techniques used.

\begin{table}
\caption{The two learning paradigms satisfying Assumptions~\ref{A:f} and \ref{A:pPhi} and the scheme in \eqref{eqn:paradigm}.}
\begin{tabularx}{\columnwidth}{l|l|l|l}
Name & $\Phi$ & Basis for convergence of $\Phi$ & $g$ \\ 
\hline
Invariant-graphs & Implicitly obtained \eqref{eqn:asdc9} & \eqref{eqn:cvd83} & Explicit : \eqref{eqn:cv93m} \\ 
\hline
Delay-coordinates & Explicitly obtained as basis functions & Ergodic convergence & Explicit : \eqref{eqn:sdjf3} \\
\hline
\end{tabularx}
\label{tab:prdgm}
\end{table}

\begin{table}
\caption{Various learning techniques, applicable to learning $w_k$ \eqref{eqn:def:wk}. }
\begin{tabularx}{\linewidth}{L|L|L|L}
Technique & Hypothesis space & Advantages & Disadvantages \\ 
\hline
Linear reg. & Linear combination of fixed basis functions or coordinates & Availability of techniques for linear cases & Poor fit for nonlinear functions \\
\hline
Kernel reg. \cite{BerryHarlim2017, BerryEtAl2015} & $C^r(M)$ or $L^2(\mu)$ Spaces spanned by kernel sections or eigenvectors & Allows smooth interpolations and connections with underlying geometry  & Localized nature of basis functions require large number of basis functions \\
\hline
RKHS \cite{AlxndrGian2020, DasGiannakis_RKHS_2018, DasDimitris_CascadeRKHS_2019} & Span of eigenfunctions of kernel integral operators & Completely data-driven, allows out of sample extension & Inexplicit, unspecified nature of basis functions  \\
\hline
Nonlinear reg. \cite{Gallant1975} & Parameterized space of functions & Dependence on parameters allow application of manifold techniques such as gradient descent & Explicit knowledge of parameters as well as dependence on parameters required \\
\hline
Deep NNs \cite{LecunEtAl2015} & Functions parameterized by network activation and coupling parameters & Simplicity of implementation; scalability; explicit dependence on parameters known & Little apriori knowledge known about dimension of layers or number of layers; little knowledge about convergence rate of learning; huge number of variables to optimize \\
\hline
LSTM \cite{HarlimEtAl2021, MaEtAl_2018, Maulik_EtAl_2020} & Same as Deep NNs but with additional memory cells & Good for approximating functions which have sparse dependence over a long interval of time & Same as Deep NNs; more parameters to tune\\
\hline
Radial basis functions \cite{Smith1992} & Similar to kernel techniques & Provides a global representation of the map & Lack of normalization lead to nonuniformity in predictability\\
\hline
Local approximation techniques, such as simplex methods \cite{JimenezEtAl1992}, and local linear regression \cite{KugiumtzisEtAl98} & Nearest-neighbor based approximation of a neighborhood of the predictee point & Good approximation for low-curvature attractors, i.e., less oscillatory functions & Predictee point needs to be close to data cloud, feedback function unbounded.\\
\hline
\end{tabularx}
\label{tab:learn}
\end{table}
%

%-_-_-_-_-_-_-_-_-_-_-_-_-_-_-_-_-_-_-_-_-_-_-_-_-_-_-_-_-_-_-_-_-_-_-_-_-_-_-_-_-_-_-_-_-_-_-_-_-_-_-_-_-_-_-_-_-_-_-_-_-_-_-_-_-_-_-_-_-_-_-_-_-_-_-_-
\subsection{Paradigm I : Invariant graphs} \label{sec:inv_graph}

Let $g:\real^d \times \real^L$ be a smooth map for which there is constant $\lambda\in (0,1)$ such that
\begin{equation} \label{eqn:cv93m}
    \norm{ \nabla_y g(u, y) } \leq \lambda, \quad \forall u\in \real^d, \, y\in \real^L.
\end{equation}
A particular instance of $g$ above introduced in \cite{LuEtAL2017} is
\[ g(u,y) = \tanh\left( W_{in} u + W_Y y + v_{bias} \right), \]
where $W_{in}$, $W_Y$ are random matrices of dimensions $L\times d$, $L\times L$ respectively, $v_{bias}$ is a random vector of dimension $L$, and $\norm{W_{Y}}\leq \lambda<1$. Using this $g$ one can build a \emph{reservoir} system, which is a skew product system on $\Omega\times\real^L$ defined as
\begin{equation} \label{eqn:def:reservoir}
\left(\begin{array}{c} \omega_{n+1} \\ y_{n+1} \end{array}\right) := T_{\text{reservoir}} \left(\begin{array}{c} \omega_{n} \\ y_{n} \end{array}\right) := \left(\begin{array}{c} f(\omega_n) \\ g\left( \phi(\omega), y_n \right) \end{array}\right) .\end{equation}
The paradigm of invariant graphs has been used in reservoir computing \citep[e.g.][]{LukoJaeger2009,  GononOrtegafading_fading_2021,  GrigoryevaHartOrtega2021}. It is popular due to the simplicity of its construction, and ease of use in learning problems. They are known for their robust performance in prediction \cite{GauthierEtAl2021, LuEtAL2017, bollt2021explaining} but also for recovering other properties such s Lyapunov exponents \cite{PathakEtAl_repl_2017}.
Although \eqref{eqn:def:reservoir} involves the underlying dynamical map $f$, the actual knowledge of $f$ is not needed. Note that the dynamics in the $y$-coordinate is linked to the $\omega$-coordinate through the measurement $\phi$. In the training phase, one provides as input the measurements $\left\{ \phi(\omega_n) \right\}_{n=0}^{N}$. Thus $(\Omega, f)$ remains unknown but continues to drive the reservoir variable $y$. The variable $y$ settles down into a representation of the attractor, in a manner which we make precise below.

\begin{proposition} \label{thm:kncd3l}
Let Assumption \ref{A:f} hold, and $\phi:\Omega\to \real^d$ be a continuous map, and $g:\real^d\times \real^L \to \real^L$ be a $C^1$ map satisfying \eqref{eqn:cv93m}. Then
\begin{enumerate}[(i)]
    \item there is a map $\Phi:\Omega \to \real^L$ such that Assumption \ref{A:pPhi}~(ii) is satisfied.
    \item The graph of $\Phi$ is invariant, i.e.,
\begin{equation} \label{eqn:asdc9}
\left( U^n\Phi \right)(\omega) := \Phi \left( f^n(\omega) \right) = \proj_Y T_{\text{reservoir}}^n \left( \omega, \Phi(\omega) \right) , \quad \forall n\in\num , \quad \forall \omega\in \Omega.
\end{equation}
\item The graph of $\Phi$ is globally attracting, i.e.,
\begin{equation} \label{eqn:cvd83}
\lim_{n\to\infty} \left( U^n\Phi \right)(\omega) = \lim_{n\to\infty} \proj_Y T^n\left( x, y \right), \quad \forall x\in \bar{\Omega}, \, y\in Y.
\end{equation}
\item If $\norm{ \frac{\partial }{\partial_u} g } \leq 1$, then  Assumption~\ref{A:g_contract} is also satisfied.
\end{enumerate}
\end{proposition}

Assumption~\ref{A:g_contract} is an additional assumption requiring that $g$ be a non-expansive map in each of its variables. It is described in Section~\ref{sec:stability} and is used to establish stability properties.  Proposition~\ref{thm:kncd3l} is proved in Section~\ref{sec:proof:kncd3l}. Parts (i)---(iii) are immediate consequences of results by Stark \cite{Stark1999}, or by Grigoryeva et. al. \citep[][Thm III.1]{GrigoryevaHartOrtega2021}. We have put Stark's results along with the other paradigms in the common, general framework of \eqref{eqn:def:wk}.

The invariant graph property leads to a fulfillment of the identity in Assumption \ref{A:pPhi}~(ii). However, the injectivity condition of Assumption \ref{A:pPhi}~(i) remains to be proven rigorously. It has been generally observed that for $L$ large enough, $\Phi$ is also injective. The ground Assumption~\ref{A:f} is assumed while running the system. Note that the embedding $\Phi$ is not obtained explicitly but implicitly through the state variables of the network. Given any arbitrary initialization to \eqref{eqn:def:reservoir}, by \eqref{eqn:cvd83}, the internal states of the reservoir converge to an invariant graph over $\Omega$. The function $\Phi$ is precisely the function whose graph is invariant. Although it will remain indeterminate, its values over a dynamic trajectory, i.e., the values $\phi(f^n \omega_0)$ will be obtained, for some unknown initial point $\omega_0$. 

\paragraph{Long-short term memory (LSTM)} LSTMs \cite{HochSchmid1997, Maulik_EtAl_2020} is a network of units, in which each unit is a skew-product system, usually much smaller in size than a reservoir network, and without the contraction requirement of a reservoir network. Each unit has internal states $y_n = (h_n, c_n)$, which is updated with the help of an additional input $x_n$ which could originate from an external dynamical system. The functional equation is
\begin{equation} \label{eqn:LSTM_unit}
    y_n = \left( h_n, c_n \right) = G\left( x_n, h_{n-1}, c_{n-1}  \right) = G\left( x_n, y_{n-1}  \right), \quad n=0,1,2,\ldots.
\end{equation}
The variables $h_n, c_n$ respectively denote a hidden state vector, and a cell input activation vector to the LSTM unit. The units in the LSTM can also be cascaded to each other Suppose there are $Q$ LSTM units. Let us denote their internal states at time $n$ as $y^{(1)}_n, \ldots, y^{(Q)}_n$. Due to the cascaded structure, we have
\begin{equation} \label{eqn:LSTM_cascade}
    y^{(q)}_n = G\left( x_{n+q-1}, y^{(q-1)}_{n-1} \right), \quad n\in \num_0, \, q\in \{1,\ldots, Q\} .
\end{equation}
Here $y^{(0)}_n$ is the constant sequence equal to zero.
LSTMs implement delay coordinates due to their full dependence on a delay-coordinate set $\left( x_{n}, \ldots, x_{n+Q-1} \right)$ for each $n\in\num_0$. In addition, these delay coordinated input is fed into a skew-product system whose internal structure is block diagonal. Note that LSTMs require tuning and do not automatically satisfy the contraction property in Assumption~\ref{A:g_contract}. It is an interesting question whether the tuning procedure with data from an ergodic trajectory  leads to this criterion being met. Alternatively, optimization methods in learning dynamics could be modified to enforce additional constraints to satisfy Assumption~\ref{A:g_contract}. These are interesting directions of future work.

%-_-_-_-_-_-_-_-_-_-_-_-_-_-_-_-_-_-_-_-_-_-_-_-_-_-_-_-_-_-_-_-_-_-_-_-_-_-_-_-_-_-_-_-_-_-_-_-_-_-_-_-_-_-_-_-_-_-_-_-_-_-_-_-_-_-_-_-_-_-_-_-_-_-_-_-
\subsection{Paradigm II : Delay coordinates} \label{sec:reg} 

An effective and numerically inexpensive means of obtaining an embedding of a dynamical system is using delay coordinates \cite{BerryEtAL2013, BerryHarlim2017}. To relate to our framework, fix a number of delays $Q\in\num$, and set 
\begin{equation} \label{eqn:sdjf3}
L = Qd, \quad \Phi : \Omega\to \real^L, \, \Phi : \omega \mapsto
\left[\begin{array}{c}
     \phi\left( \omega \right) \\
     \vdots \\
     \phi\left( f^{Q-1} \omega \right) 
\end{array}\right]
,\quad g : \real^d\times \real^L \to \real^L, \, g: 
u\times \left[\begin{array}{c}
     y^{(1)} \\
     y^{(2)} \\
     \vdots \\
     y^{(Q)} 
\end{array}\right]
\mapsto
\left[\begin{array}{c}
     u \\
     y^{(1)} \\
     \vdots \\
     y^{(Q-1)} 
\end{array}\right]
\end{equation}
Using this paradigm the reconstructed dynamics \eqref{eqn:def:feedback_1} becomes :
\[ \reconstruct_{\text{delay-coord}} :\real^{d\times dQ} \to \real^{d\times dQ} := \left[\begin{array}{c}
     u \\
     y^{(1)} \\
     y^{(2)} \\
     \vdots \\
     y^{(Q)} 
\end{array}\right] \mapsto  
\left[\begin{array}{c}
    w\left( y^{(1)}, \ldots, y^{(Q)} \right) \\
     y^{(1)} \\
     \vdots \\
     y^{(Q-1)} 
\end{array}\right]
\]
Note that in this case, $g$ is a linear map. We have :

\begin{proposition} \label{prop:dwa4ep}
Let Assumption~\ref{A:f} hold. Then for a typical map $\phi:\Omega\to \real^d$, if $Q\in\num$ is large enough, then $\Phi$ defined through \eqref{eqn:sdjf3} is an embedding, and thus Assumption~\ref{A:pPhi} is satisfied. Moreover, Assumption~\ref{A:g_contract} is also satisfied.
\end{proposition}

The proof is a direct consequence of the delay coordinate embedding theorem \cite{SauerEtAl1991}.

Thus the most common techniques for reconstruction and forecasting fall into the framework we introduced in \eqref{eqn:paradigm}. See Figures \ref{fig:compare} and \ref{fig:T2_L63Rot} for an illustration of application of these two techniques. In Section~\ref{sec:Koop_approx}, we also briefly review some techniques which do not fall under the schemes of Assumption~\ref{A:pPhi} and \eqref{eqn:paradigm}. In the next two sections, we shall analyze two important features of our scheme, their stability, and the accuracy of their predictions.

%-_-_-_-_-_-_-_-_-_-_-_-_-_-_-_-_-_-_-_-_-_-_-_-_-_-_-_-_-_-_-_-_-_-_-_-_-_-_-_-_-_-_-_-_-_-_-_-_-_-_-_-_-_-_-_-_-_-_-_-_-_-_-_-_-_-_-_-_-_-_-_-_-_-_-_-
\section{Stability of reconstructed system} \label{sec:stability}

Although the feedback system \eqref{eqn:def:feedback_1} is a conjugate to the original dynamics $f:\Omega\to \Omega$, as shown in \eqref{eqn:feedback_commut}, one needs to consider it stability properties. The image of $ h := \phi\times \Phi$ is a bijective image of $\Omega$, and is invariant under the dynamics of $\reconstruct$ \eqref{eqn:def:feedback_1}. But $\reconstruct$ acts in the higher dimensional ambient space $\real^{L+d}$, and one needs to calculate the rate of deviation perturbations from $X := h(\Omega)$. We track these using Lyapunov exponents. Let the distinct Lyapunov exponents of $(\Omega, f, \mu)$ be $\lambda_1 > \lambda_2 \cdots > \lambda_r$, with corresponding Oseledet splitting $T\Omega = E_1 \oplus \cdots \oplus E_r$. Since the dimension of $\tilde\Omega$ is $m$, the multiplicities of the $\lambda_i$ sum to $m$. The $E_i$s corresponding to negative valued $\lambda_i$ constitute the stable directions, whereas the $E_i$ corresponding to positive valued $\lambda_i$ constitute the unstable directions. Moreover,
\[ \lim_{n\to\infty} \frac{1}{n} \ln \norm{Df^n(\omega) v_i} = \lambda_i, \quad \mu-a.e. \omega\in \Omega, \, \forall v_i\in E_i(\omega)\setminus\{0\} . \]
In general, the map $\reconstruct$ will have $L+d$ Lyapunov exponents (counting multiplicities), whereas $f$ has $m$ of them. We shall show in Theorem~\ref{thm:lambda1}(i) that $m$ of the $d+L$ Lyapunov exponents of $\reconstruct$ coincide with the original $d$ coefficients. We are interested in these other $d+L-m$ Lyapunov exponents of the reconstructed systems, and their positions relative to $\lambda_1(f), \ldots, \lambda_d(f)$. These additional Lyapunov exponents have been labeled as \emph{spurious Lyapunov exponents} \citep[e.g.][]{SauerEtAl_spurious_1998, DechertGencay2000}. They pose significant challenges in data-driven identification of true Lyapunov exponents.

By the very definition of Lyapunov exponents, the $\lambda_i(\reconstruct)$ depend not only on the invariant set $X$ but also on its neighborhood. This leads to a problem of ambiguity. An essential part of $\reconstruct$ is the feedback function $w$. The function $w : \real^L \to \real^d$ is defined uniquely only on $X$. The conjugacy in \eqref{eqn:feedback_commut} will be preserved on $\Omega$ irrespective of the nature of the extension of $w$ to a neighborhood of $X$. We define a collection 
\[ \W := \SetDef{ \hat{w} \in C^1\left( \real^L; \real^d \right) }{  \hat{w} \rvert_{X} = w\rvert_{X} }, \]
equipped with the $C^1$-topology. Every $\hat{w}\in \W$ is a $C^1$ function satisfying $\hat{w}\circ\Phi(\omega) = (U\phi)(\omega)$ for every $\omega\in \Omega$. Any choice of $\hat{w}\in \W$ leads to a different dynamics in $\real^{L+d}$ as in \eqref{eqn:def:feedback_1}, with $X$ as an invariant ergodic set. Therefore the top Lyapunov exponent $\lambda_1$ of \eqref{eqn:def:feedback_1} will be a function of $\hat{w}$. Thus we can define a function :
\[%\begin{equation} \label{eqn:def:lambda1}
\lambda_1 : \W\to \real, \quad \lambda_1(\bar{w}) := \lambda_1( \reconstruct ).
\]%\end{equation}
Our goal will be to study how close $\lambda_1(\reconstruct)$ can be made to $\lambda_1(f,\mu)$.

\paragraph{Stability gap} As pointed out in \cite{DechertGencay2000, DechertGencay1996}, the top Lyapunov exponent of the reconstructed system may exceed that of the original system. Moreover, some of the additional Lyapunov exponents may be positive. All these contribute to additional instabilities being introduced into the system. The \emph{stability gap} in the reconstruction of a dynamical system $(\Omega, \mu, f)$ is defined to be
\[ \text{stability gap} :=\inf_{\bar{w}\in \W} \lambda_1(\bar{w}) - \lambda_1(f,\mu) . \]
The stability gap is always non-negative, as will be shown in Theorem~\ref{thm:lambda1}. We shall study ways to obtain a bound on the stability gap in our next theorem. We shall need two additional assumptions.

\begin{Assumption} \label{A:g_contract}
The function $g$ from Assumption~\ref{A:pPhi} further satisfies :
\[ \sup_{\omega\in\Omega} \norm{\partial_1 g} \rvert_{(\phi(\omega), \Phi(\omega))} \leq 1 , \quad \sup_{\omega\in\Omega} \norm{\partial_2 g} \rvert_{(\phi(\omega), \Phi(\omega))} \leq 1 . \]
\end{Assumption}

Our next assumption requires a retraction to the range of $\Phi$. Let $\mathcal{U}$ be a neighborhood of $\ran \Phi$. Recall that a continuous map $\retract : \mathcal{U} \to \ran \Phi$ is said to be a \emph{retract} if $\retract|_{ \ran \Phi } = \Id_{ \ran \Phi }$.

\begin{Assumption} \label{A:w_ext}
There is a continuous retraction $\retract : \mathcal{U} \to \ran \Phi$, for some open neighborhood $\mathcal{U}$ of $\ran \Phi$ in $\real^L$.
\end{Assumption}

For a retraction map such as $\retract$, we shall be interested in the Lipschitz norm of the retraction
\begin{equation} \label{eqn:def:kappa_ret}
\kappa_{\retract} := \sup_{y\in \ran\Phi} \limsup_{y'\to y} \frac{ d\left( \retract(y), \retract(y') \right) }{ d\left( y, y'\right) } .
\end{equation}
Finally, we also need the following function $\Csens$ that depends on the (fixed) functions $\phi, \Phi$ and a point $\omega\in \Omega$.
\[ \Csens : \Omega \to \real^+, \quad \Csens(\omega) := \sup \SetDef{ \frac{ \norm{D\phi(\omega) v } }{ \norm{D\Phi(\omega) v } } }{ v\in T_\omega \Omega \setminus \{0\} }. \]
We shall use $\Csens$ to bound the gap between between $\lambda_1(\reconstruct)$ and $\lambda(f)$.

\begin{theorem}[Stability of reconstruction] \label{thm:lambda1}
Let Assumptions~\ref{A:f} and \ref{A:pPhi} hold. Then
\begin{enumerate} [(i)]
    \item The $d+L$ Lyapunov exponents of $\reconstruct$ contains as a subset the $m$ Lyapunov exponents of $f$.
    \item $\lambda_1(\bar{w})$ is upper semi-continuous with respect to $\bar{w}$. In other words, for every $\epsilon>0$, there is a $C^1$ neighborhood $\mathcal{U}$ of $\bar{w}$ such that
    \[ \lambda_1( \bar{w'} ) < \lambda_1(\bar{w}) + \epsilon, \quad \forall \bar{w'} \in \mathcal{U} . \]
    \item Suppose Assumptions~\ref{A:g_contract} and \ref{A:w_ext} also hold. Then the stability gap is bounded by
    \begin{equation} \label{eqn:smf03}
    \inf_{\bar{w}\in \W} \lambda_1(\bar{w}) - \lambda_1(f,\mu) \leq \int \ln \left[ 1 + \left(1 + \Csens(\omega) \right) \kappa_{\retract} \right] d\mu(\omega) .
    \end{equation}
\end{enumerate}
\end{theorem} 

Claims (i) and (ii) of Theorem~\ref{thm:lambda1} are immediate consequences of results from \cite{DechertGencay1996, BochiViana_conti_2005, Viana_Lyap_2020}. We review these and prove Claim~(iii) in Section~\ref{sec:proof:lambda1}. 

\begin{remark} [Instability of the reconstructed system] \label{rem:instab}
Claims (ii) and (iii) imply that at least in theory, given any bound $\epsilon$, there is a robust (i.e. open) set of $\bar{w}$ for which the instability is no more than $C+\epsilon$ of the original dynamics, for some constant $C$ depending on the dynamics, $\phi$ and $\Phi$ alone. In practice, $\hat{w}$ is obtained from some hypothesis space which is determined by the application domain. In such situations there is no guarantee of the stability being preserved up to an $\epsilon$ error.
\end{remark}

\begin{remark} [Continuity of Lyapunov exponents] \label{rem:cont_Lyap} 
Theorem~\ref{thm:lambda1} is related to the important question of continuity of Lyapunov exponents. In our case, we show that the growth of error in the system is related to a $GL(2L)$-valued matrix cocycle over the base dynamics $(f,\mu,\omega)$, described in detail in \eqref{eqn:ab_exact}. These cocycles are dependent (in a $C^1$-sense) on $\bar{w}$. Thus a relevant question for us is the continuity of $\lambda_1$ for these cocycles, as a function of $\bar{w}$. There has been various results in this direction, such as for iid matrix cocycles \cite{Furstenberg_Kifer_1983, Furstenberg_Kesten_1960}, in terms of \emph{large deviation type parameters} \citep[see][Thm 1.6]{Duarte_Klein_2016}, and in terms of dominated splittings \citep[][Thm 5]{BochiViana_conti_2005}. See \cite{Viana_Lyap_2020} for a broad overview of this extensive field of investigation. However, none of these various set of assumptions apply to our situation, in which the cocycle family is parameterized by a set of functions $\W$.
\end{remark}

Assumption~\ref{A:w_ext} is of a topological nature and would depend on the topological or geometrical properties of $X$. The following corollary applies to the use of a large number of delay-coordinates. %\cite{DechertGencay1996, DechertGencay2000} that using leads to the top Lyapunov exponents of the reconstructed system staying close to the true Lyapunov exponents.

\begin{corollary} \label{cor:delay_stab}
Let $\Psi^t : \Omega\to \Omega$ be a smooth flow and $f$ be the time-$\Delta t$ map $f=\Psi^{\Delta t}$. Let all the conditions in Assumptions \ref{A:f} and \ref{A:pPhi} be met and the delay coordinate paradigm \eqref{eqn:sdjf3} be implemented. Suppose further that there is a retraction map as in Assumption~\ref{A:w_ext} for which the Lipschitz constant $\kappa_{\retract} = 1$. Then there is a constant $C_2$ depending only on the flow such that $\Csens(\omega) \leq \frac{1}{Q} + 0.5 C_2 Q \Delta t$ for every $\omega\in \Omega$. In particular, 
\[ 0\leq \inf_{\bar{w}\in \W} \lambda_1(\bar{w}) - \lambda_1(f,\mu) \leq \ln\left[ 2+ \frac{2}{Q} + C_2 Q\Delta t \right]. \]
\end{corollary}

The criterion that $\kappa_{\retract} = 1$ is attained for example when $X = \ran\Phi$ is a manifold, and $\retract$ is a tubular neighborhood retract. Corollary~\ref{cor:delay_stab} is proved in Section~\ref{sec:proof:delay_stab}. 

Next, we analyze the divergence of the dynamics of $\hat{\reconstruct}$ \eqref{eqn:def:feedback_2} from that of the perfect reconstruction $\reconstruct$ \eqref{eqn:def:feedback_1}.

%-_-_-_-_-_-_-_-_-_-_-_-_-_-_-_-_-_-_-_-_-_-_-_-_-_-_-_-_-_-_-_-_-_-_-_-_-_-_-_-_-_-_-_-_-_-_-_-_-_-_-_-_-_-_-_-_-_-_-_-_-_-_-_-_-_-_-_-_-_-_-_-_-_-_-_-
\section{Forecasts with reconstructed system} \label{sec:forecast}

We shall now analyze the effectiveness of forecasts made using the scheme in \eqref{eqn:paradigm}, and its approximation as \eqref{eqn:def:feedback_2}. As suggested by Casdagli \cite{Casdagli1989}, given a reconstruction paradigm, there are two ways of estimating the value $\phi(f^n \omega)$ after $n$ iterations of the base dynamics : we can iterate \eqref{eqn:def:feedback_2} $n$ times, and the first coordinate of $z_0$ will serve as an approximation of $\phi( f^n \omega) = (U^n \phi)(\omega)$. We call this the iterative method, and its accuracy can be estimated via
\begin{equation} \label{eqn:def:err_iter}
\begin{split}
\text{error}_{\text{iter}}(n, \omega) &:= \norm{ U^n \phi (\omega) - \proj_1 \circ \hat{\reconstruct}^n \circ (\phi, \Phi)(\omega) }_{\real^d} , \\
\text{error}_{\text{iter}}(n) &:= \left[ \int_\Omega \text{error}_{\text{iter}}(n, \omega)^2 d\mu(\omega) \right]^{1/2}.
\end{split}
\end{equation}
Or we can directly approximate $w_n$ via \eqref{eqn:def:hat_w} and obtain a \emph{direct} estimate. The corresponding errors are
\begin{equation} \label{eqn:def:err_direct}
\begin{split}
\text{error}_{\text{direct}}(n, \omega) &:= \norm{ U^n \phi (\omega) - \hat{w}_n \circ \Phi (\omega) }_{\real^L} ,\\
\text{error}_{\text{direct}}(n) &:= \norm{ U^n \phi - \hat{w}_n \circ \Phi }_{L^2(\mu)} = \left[ \int \text{error}_{\text{direct}}^2 (n, \omega) d\mu(\omega) \right]^{1/2}.
\end{split}
\end{equation}
To aid the discussion, we will make further assumptions on the nature of the hypothesis space $\mathcal{H}$.

\paragraph{Linear hypothesis space} Usually the hypothesis space $\mathcal{H}$ will be a finite dimensional space, spanned by a basis $h_1, \ldots, h_m$. In that case
\begin{equation} \label{eqn:def:W}
\mathcal{W} := \spn\SetDef{ h_i \circ \Phi_l }{ 1\leq i\leq m, \, \leq l\leq L }
\end{equation}
is a finite subspace of $L^2(\mu)$, and 
\begin{equation} \label{eqn:kcnv83}
\hat{w}_k \circ \Phi = \proj_{\mathcal{W}} U^k \phi .
\end{equation}
For example, if the hypothesis space is restricted to $\Linear\left( \real^L; \real^d \right)$, then $\mathcal{W} = \spn \Phi$. In the rest of this paper, we shall focus on this scenario where the hypothesis space is linear. We state this formally :

\begin{Assumption} \label{A:hypoth}
The hypothesis space $\mathcal{W}$ is a finite dimensional subspace of $L^2(\mu)$, and contains the constant function $1_{\real^L}$.
\end{Assumption}

In most learning techniques, a bias or offset constant is calculated separately, thus satisfying the criterion that $\mathcal{W}$ contains constant functions.

Let $\pi$ denote the projection $\proj_{\mathcal{W}}$, and set $\Delta := \Id-\pi$. For ease of notation, we will denote $\hat{w}_1$ simply by $\hat{w}$, in the rest of this section. Define the \emph{projection error}  to be the quantity
\begin{equation} \label{eqn:def:delta}
\delta = \delta(\mathcal{H}) := \norm{ \Delta U\phi }_{L^2(\mu)} .
\end{equation}
This is the component of the measurement $\phi$ not recoverable using our choice of hypothesis space. Note that as the size of the hypothesis space increases, $\delta$ converges to $0$. 

We shall first examine the performance of the direct forecast method. For this purpose, we shall utilize a natural splitting of the space $L^2(\mu)$ induced by the Koopman operator. Let $\Disc$ be the closure of the span of the eigenfunctions of the Koopman operator $U$, and $\Disc^\bot$ be its orthogonal complement. Thus we have the orthogonal splitting
\[ L^2(\mu) = \Disc \oplus \Disc^\bot .\]
The space $\Disc$ always contain the constant functions. for mixing systems such as the Lorenz 63 attractor, $\Disc$ contains only the constant functions. For quasiperiodic dynamics such as the dynamics on Hamiltonian tori, $\Disc^\bot = \{0\}$. These two components $\Disc, \Disc^\bot$ not only have different ergodic properties \cite{Halmos1956, DasJim2017_SuperC}, but also respond differently to data-analytic and Harmonic analytic tools \cite{DasGiannakis_delay_2019, DasGiannakis_RKHS_2018}. This splitting is also natural in the sense that it is invariant under the action of the Koopman operator. We now provide an estimate on the rate of growth of the direct error.

\begin{theorem}[Error from direct forecast] \label{thm:direct}
Let Assumptions~\ref{A:f} and \ref{A:pPhi} hold, and assume the notations in \eqref{eqn:def:hat_w}, \eqref{eqn:def:feedback_1} and \eqref{eqn:feedback_1_init}. Let $\delta$ be as in \eqref{eqn:def:delta}. Then the error from direct iteration is given by
\[ \text{error}_{\text{direct}}(n) = \norm{ \left( \Id - \pi \right) U^n\phi }_{L^2(\mu)} . \]
Now assume that Assumption~\ref{A:hypoth} holds. Then there is a subset $\num'\subseteq \num$ with density $1$ such that the following holds : 
\begin{enumerate} [(i)]
    \item For every $\epsilon>0$, if the hypothesis space $\mathcal{W}$ is chosen large enough, then
    \[ \lim_{n\in \num', n\to\infty} \text{error}_{\text{direct}}(n) = \norm{ \phi - \proj_{\Disc} \phi }_{L^2(\mu)} + \epsilon. \]
    \item If $f$ is weakly mixing, then for every choice of $\mathcal{W}$
    \[ \lim_{n\in \num', n\to\infty} \text{error}_{\text{direct}}(n) = \text{var}_{\mu} := \norm{ \phi - \mu(\phi)}_{L^2(\mu)} . \]
    \item If $f$ is strongly mixing, the set $\num'$ can be taken to be the entire set $\num$.
    \item If $f$ has purely discrete spectrum, then for every $\epsilon>0$, if the hypothesis space $\mathcal{W}$ is chosen large enough, then
    \[ \text{error}_{\text{direct}}(n) < \epsilon, \quad \forall n\in\num . \]
\end{enumerate}
\end{theorem}

Theorem~\ref{thm:direct} is proved in Section~\ref{sec:proof:direct}. An important basis for claims (i), (ii) is the decay of correlations seen in (weakly) mixing systems. See Remark~\ref{rem:correl} for further discussions on this topic. 
We next study the performance of the iterative method. It will be stated in terms of a construct called \emph{matrix cocycle}s.

\paragraph{Associated matrix cocycle} Matrix cocycles over the dynamics $(\Omega, \mu, f)$ will be defined in more generality later in Section~\ref{sec:cocyc:basics}. For the moment, we focus on the following matrix valued functions
\begin{equation}\begin{split} \label{eqn:def:WG_matrices}
W :\Omega\to \real^{d\times L}, & \quad W(\omega) := Dw \circ \Phi(\omega) ,\\
\hat{W} :\Omega\to \real^{d\times L}, & \quad \hat{W} (\omega) := D\hat{w} \rvert_{ \Phi(\omega)} = D\hat{w} \circ \Phi(\omega),\\
G^{(1)} :\Omega\to \real^{L\times d}, & \quad G^{(1)}(\omega) := \nabla_1 g \rvert_{h(\omega)} = \nabla_1 g \circ h(\omega),\\ 
G^{(2)} :\Omega\to \real^{L\times L} , & \quad G^{(2)}(\omega) := \nabla_2 g \rvert_{h(\omega)} \nabla_2 g \circ h(\omega),
\end{split}\end{equation}
and their combination 
\begin{equation} \label{eqn:def:M_matrix}
M : \Omega\to \real^{(L+d)\times (L+d)}, \quad M(\omega) := \left[\begin{array}{cc} 0^{d\times d} & W(\omega) \\ G^{(1)}(\omega) & G^{(2)}(\omega)  \end{array}\right] .
\end{equation}
Next consider the vector-valued functions
\[ c:\Omega\to \real^L, \, c(\omega) := G^{(1)}(\omega) \left( U^{-1} \Delta \phi \right)(\omega). \]
%\[ c:\Omega\to \real^d, \, c(\omega) := \left( U^{n-1} \Delta U\phi \right)(\omega) , \quad d:\Omega\to \real^L, \, d(\omega) := \left[ 0^d \\ \begin{array}{c} G^{(1)}(\omega) c(\omega) \end{array}\right] . \]
%
We shall use this to build a non-autonomous dynamical system on $\real^{d+L}$. Fix an $\omega\in \Omega$ and define
\begin{equation}  \label{eqn:ab_exact}
\left[\begin{array}{c} a_{n+1} \\ b_{n+1} \end{array} \right] = M \left( f^n \omega \right) \left[\begin{array}{c} a_{n} \\ b_{n} \end{array} \right] + \left[\begin{array}{c} 0 \\ c\left( f^n\omega \right) \end{array} \right] , \quad a_1=0^d, \; b_0=0^L .
\end{equation}
We call such a system a \emph{perturbed matrix cocycle} [see Section~\ref{sec:cocyc:peturb}]. Note that as the size of the hypothesis space is increased, the function $c$ converges to $0$ in $L^2(\mu)$ norm, and the dynamics of $(a_n, b_n)$ gets closer to that of the matrix cocycle generated by $M$. We shall examine this closely in Theorem~\ref{thm:GraphT_growth}. Note that \eqref{eqn:ab_exact} depends on the initial state $\omega$. If $\omega$ is allowed to vary, then $a_n, b_n$ become functions of $\omega$. We shall overuse notation and also denote these functions as $a_n, b_n$. 

\paragraph{Growth of the iterative error} In Theorem~\ref{thm:iterative} below, we shall establish a rate at which the iterative error grow. Let $(u_n, y_n)$ be iterates of the system \eqref{eqn:def:feedback_1}. We are    interested in the growth of the deviation quantities $\Delta u_n, \Delta y_n$ defined as
\begin{equation} \label{eqn:def:Delta_uy}
\left[\begin{array}{c} \Delta u_n \\ \Delta y_n \end{array} \right] = \left[\begin{array}{c} U^{n-1} \pi U \phi \\ U^n\Phi \end{array} \right] - \left[\begin{array}{c} u_n \\ y_n \end{array} \right], \quad \forall n\in\num_0.
\end{equation}
Note that when defining the deviation terms, we are using as reference the functions $U^{n-1}\pi U$ and $U^n \Phi$, both of which reflect the true state of the dynamics $(\Omega, f)$. See Remark~\ref{rem:Ukpi} for further discussions on their significance.
We now have

\begin{theorem}[Error from iterative forecast] \label{thm:iterative}
Let Assumptions~\ref{A:f} and \ref{A:pPhi} hold, and assume the notations in \eqref{eqn:def:hat_w}, \eqref{eqn:def:feedback_1} and \eqref{eqn:feedback_1_init}. Fix an initial state $\omega\in \Omega$ and let $(u_n, y_n)$ be successive iterations of the system \eqref{eqn:def:feedback_1}, while $(a_n, b_n)$ be iterations of the dynamics in \eqref{eqn:ab_exact}. 
\begin{enumerate} [(i)]
    \item Let $\delta$ be as in \eqref{eqn:def:delta}. The deviations \eqref{eqn:def:Delta_uy} have the following relations with the states of the associated perturbed cocycle :
    \begin{equation} \label{eqn:u_vs_a}
    \Delta u_n = a_n + \bigO{a_{n-1}}^2, \quad \lim_{\delta\to 0} \frac{ \norm{\Delta u_n} }{ \norm{a_n} } = 1 .
    \end{equation}
    \item Let $\lambda_1 = \lambda_1(\mathcal{M})$  the maximal Lyapunov exponent of the cocycle generated by $\hat{M}$. Then for every $\epsilon>0$, there is a a constant $C^{(1)}_{\omega, \epsilon}$ such that
    \begin{equation} \label{eqn:def:iter_bound}
    \text{error}_{\text{iter}}(n, \omega) = \norm{\Delta u_n(\omega)}_{\real^L} = \delta C^{(1)}_{\omega, \epsilon} \bigO{ e^{n (\lambda_1+\epsilon)} }, \quad \mbox{as } n\to\infty ,
    \end{equation}
    \item If $(\Omega, \mu, f)$ has the additional property of $L^2$ Pesin sets, then for every $\epsilon>0$,
    \begin{equation} \label{eqn:def:iter_bound_L2}
    \text{error}_{\text{iter}}(n) = \norm{\Delta u_n}_{L^2(\mu)} = \delta C^{(2)}_{\epsilon} \bigO{ e^{n (\lambda_1+\epsilon)} }, \quad \mbox{as } n\to\infty .
    \end{equation}
    for a constant $C^{(2)}_{\epsilon}$ that depends only on $\epsilon$.
\end{enumerate}
\end{theorem}

Pesin sets are subsets of $\Omega$ on which the non-uniformly hyperbolic map $f$ has some degree of regularity. While Pesin sets always exist and cover the entire space $\Omega$, the property of $L^2$ Pesin sets is an additional property, explained in more details in Section~\ref{sec:cocyc:pesin}. Theorem~\ref{thm:iterative} is proved in Section~\ref{sec:proof:iterative}.

\begin{remark}[$U^{n-1}\pi U$ vs $U^n \pi$] \label{rem:Ukpi} The explicit formulas for the direct and iterative schemes reveal a basic mathematical law that makes the direct method unsuitable for long term prediction. It involves the operator $\pi U^n$, which projects the evolving measurement $\phi$ back into the space $\mathcal{W}$. For strongly mixing systems, the Koopman operator drives out any function from any finite dimensional subspace, up to a constant function. On the other hand, the iterative method always involves the term $U^{n-1} \pi U$.  The crucial difference is that the projection $\pi$ is not made after the application of $U^n$, but always to the static operator $U$. The $U^{n-1}$ in front of the $\pi$ then merely acts as a rotation / unitary transform and thus does not change the $L^2(\mu)$ (i.e., RMS) magnitude of the error. Also see Remark~\ref{rem:correl} for a further discussion on decay of correlations.
\end{remark}

\begin{remark} [Decay of correlations] \label{rem:correl} Theorem~\ref{thm:direct} relates the growth of $\text{error}_{\text{direct}}$ with the rate of decay of correlation, while Theorem~\ref{thm:iterative} relates the growth of $\text{error}_{\text{iter}}$ with the top Lyapunov exponent. The former is a spectral / operator theoretic property, while the latter is a combination of differential and ergodic properties such as Lyapunov exponents. The exact connections between mixing and positive Lyapunov exponents are still far from understood \cite{AlvesEtAl2005}. Connections has been established heuristically in some cases such as \cite{LogButk2008, ColletEck2004}, and rigorously under additional assumptions such as the existence of 
finite Markov partitions \cite{Young_mixing_1999, Young_correl_1998}, or expansive property of the map \cite{AlvesEtAl_mixing_2004}.
\end{remark}

\begin{remark} [Autocorrelations] \label{rem:autocorr} Given any nonzero function $\psi \in L^2(\mu)$, we define its normalized autocorrelation (with respect to the underlying dynamics) as
\[ \AutCor(n; \psi):= \norm{\psi}^{-2} \left\langle U^n \psi, \psi \right\rangle . \]
Now, suppose that $\psi$ lies in $\mathcal{W}$. Let $\{ \mathfrak{w}_i \}_{i=1}^{M}$ be any orthonormal basis for the hypothesis space $\mathcal{W}$. Then
\[\begin{split}
    \AutCor(n; \psi)^2 &:= \norm{\psi}^{-4} \abs{ \left\langle U^n \psi, \psi \right\rangle }^2 = \norm{\psi}^{-4} \abs{ \left\langle U^n \psi, \sum_i \left\langle \mathfrak{w}_i, \psi \right\rangle \mathfrak{w}_i \right\rangle }^2 = \norm{\psi}^{-4} \abs{ \sum_{i=1}^{M} \left\langle \mathfrak{w}_i, \psi \right\rangle \left\langle U^n \psi, \mathfrak{w}_i \right\rangle }^2 \\
    &\leq \norm{\psi}^{-4} \sum_{i=1}^{M} \abs{ \left\langle \mathfrak{w}_i, \psi \right\rangle }^2 \sum_{i=1}^{M} \abs{  \left\langle U^n \psi, \mathfrak{w}_i \right\rangle }^2 = \norm{\psi}^{-2} \sum_{i=1}^{M} \abs{  \left\langle U^n \psi, \mathfrak{w}_i \right\rangle }^2 .
\end{split}\]
We show in Section~\ref{sec:proof:direct} that 
\[\text{error}_{\text{direct}}(n)^2 = \norm{ \left( \Id - \pi \right) U^n\phi }_{L^2(\mu)}^2 = \norm{\phi}^2 + \norm{\pi U^n \phi}^2 - 2\left\langle U^n \phi, \pi U^n \phi \right\rangle = \norm{\phi}^2 - \sum_{i=1}^{M} \abs{ \langle U^n\phi, \mathfrak{w}_i \rangle }^2. \]
Combining, we get
\begin{equation} \label{eqn:phi_hypo_autocorr}
    \phi\in\mathcal{W} \imply \text{error}_{\text{direct}}(n)^2 \leq \norm{\phi}^2 \left[ 1 - \AutCor(n; \phi)^2 \right].
\end{equation}
Thus, if the hypothesis space happens to include the initial observation map $\phi$, then the growth of the direct error is directly related to the autocorrelation function of the observed signal $\phi$. Autocorrelation is a statistical property of signals used frequently in classical timeseries analysis \citep[e.g.][]{box2015time}. Equation~\ref{eqn:phi_hypo_autocorr} thus combines concepts from learning theory, ergodic theory and timeseries analysis.
%\[\text{error}_{\text{direct}}(n)^2 =  \norm{\phi}^2 + \sum_i \abs{ \langle U^t\phi, w_i \rangle }^2 - 2 \sum_i \abs{ \langle U^t\phi, w_i \rangle }^2 = \norm{\phi}^2 - \sum_i \abs{ \langle U^t\phi, w_i \rangle }^2 \]%the second term decays due to mixing, and we are left with the variance of the signal. 
\end{remark}

\begin{remark}[Overfitting error vs projection error] \label{rem:overfit} 
Equation~\eqref{eqn:def:iter_bound} shows that the projection rate grows exponentially as expected from the presence of a Lyapunov exponent. The rate of growth is proportional to the smoothness of the learnt function $\hat{w}_1$, while the multiplicative constant is proportional to the projection error $\delta$. Thus this displays a trade-off between projection error and overfitting, one can minimize the projection error by increasing the hypothesis space. But the resulting learnt function may be too oscillatory, as a result increasing the instability of the feedback system \eqref{eqn:def:feedback_1}. On the other hand, if one approximates $w_1$ by a less oscillatory function, our base error $\delta$ itself will be large to begin with. To state this trade-off more precisely, define
\[ \theta(\epsilon) := \inf \SetDef{ \norm{D\hat{w}} }{ \hat{w} \in \mbox{ some hypothesis space } \mathcal{H}, \, \delta(\mathcal{H}) < \epsilon  } .\]
Then $\theta$ is a non-decreasing function of $\epsilon$, satisfying
\[ \theta\left( \norm{\phi} \right) = 0, \quad \lim_{\epsilon\to 0^+} \theta(\epsilon) = \norm{Dw}. \]
Thus \eqref{eqn:def:iter_bound} can be rewritten as 
\[ \text{error}_{\text{iter}}(k) = \epsilon \bigO{ k \theta(\epsilon)^k }, \quad \mbox{as } k\to\infty. \]
\end{remark}

\begin{remark}[Cocycle structure] \label{rem:cocyc}
Equations \eqref{eqn:def:iter_bound} and \eqref{eqn:ab_exact} together describe the evolution of the reconstructed dynamics as the normal and error parts respectively. %Remarks \ref{rem:Ukpi} and \ref{rem:correl} provide interpretations of \eqref{eqn:def:iter_bound}. 
The format of \eqref{eqn:ab_exact} is as a \emph{matrix cocycle}, one of the major contributions of our paper. It also bears similarity with the \emph{perturbed non-autonomous equations}, studied in the continuous-time case by Barreira and Valls 
\cite{BarreiraValls2006, BarreiraValls2005}. We look more closely at the growth or decay of these cocycles in Section~\ref{sec:cocyc:peturb} and Theorem \ref{thm:GraphT_growth}.
\end{remark}

\begin{remark}[Tightness of bounds] \label{rem:tighness} The bound derived in \eqref{eqn:def:iter_bound} is not a tight bound. We obtain a better estimate in Section~\ref{sec:proof:iterative} in terms of the full Lyapunov spectrum.
\end{remark}

This completes the statement of our main results. In Section~\ref{sec:discuss}, we discuss the consequences of our results, and look at some numerical verification. In Section~\ref{sec:cocyc} we review some concepts from random matrix cocycle theory. In the sections after that, we prove our theorems.

%-_-_-_-_-_-_-_-_-_-_-_-_-_-_-_-_-_-_-_-_-_-_-_-_-_-_-_-_-_-_-_-_-_-_-_-_-_-_-_-_-_-_-_-_-_-_-_-_-_-_-_-_-_-_-_-_-_-_-_-_-_-_-_-_-_-_-_-_-_-_-_-_-_-_-_-
\section{Examples and discussions} \label{sec:discuss}

In this section we explore some of the consequences of Theorem~\ref{thm:direct} and Theorem~\ref{thm:iterative}.
\begin{enumerate}
\item According to Theorem~\ref{thm:direct}~(iii), for a weakly mixing system, the direct prediction loses track of the signal and eventually only retains the mean of $\phi$. The error from direct prediction thus converges to the variance of the observation $\phi$.
\item For a mixed spectrum system, the direct method should recover a portion of $\proj_{\Disc} \phi$, depending on the size of the hypothesis space $\mathcal{H}$, and lose track of the complementary component. Moreover, in equations \eqref{eqn:adjp30}, \eqref{eqn:lmd93c} we derive later, the error $\text{error}_{\text{direct}}(n)$ does not converge to the variance, but fluctuates periodically.
\item The growth of the iterative error on the other hand does not depend on spectral properties of the dynamics. It depends on the top Lyapunov exponent $\lambda_1(\hat{w})$ of the reconstructed dynamics, which in turn depends on the top Lyapunov exponent of the original dynamics $(\Omega, f, \mu)$ as well as the accuracy of the approximation $\hat{w}$. In a practical application, $\lambda_1(\hat{w})$ could be affected by the number of training data, and the smoothness of the true feedback function $w$.
\item Another feature of the iterative error is that unlike the direct error, it is not bounded by $\norm{ \phi }_{L^2(\mu)}$ as it not the result of the applications of pure operators, but it is the deviation between the trajectories of two different dynamical systems. Thus the error could be of the order of $\sqrt{2} \norm{\phi}_{L^2(\mu)}$.
\end{enumerate}

We next describe some numerical experiments conducted to verify and illustrate these universal behaviors.

%-_-_-_-_-_-_-_-_-_-_-_-_-_-_-_-_-_-_-_-_-_-_-_-_-_-_-_-_-_-_-_-_-_-_-_-_-_-_-_-_-_-_-_-_-_-_-_-_-_-_-_-_-_-_-_-_-_-_-_-_-_-_-_-_-_-_-_-_-_-_-_-_-_-_-_-
\subsection{Numerical experiments}

\begin{figure}[!ht]\center
\includegraphics[width=.48\linewidth]{\figs 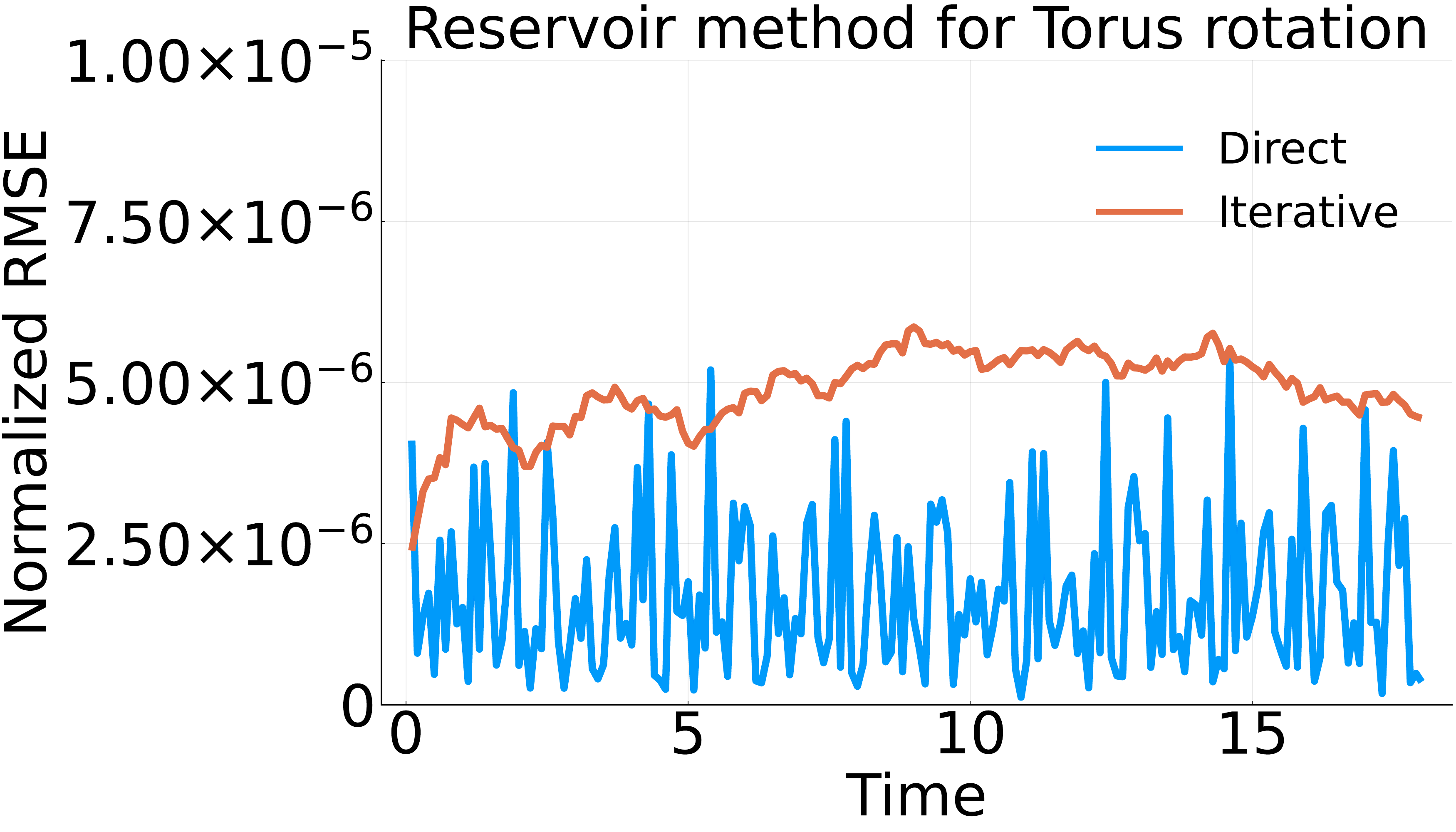}
\includegraphics[width=.48\linewidth]{\figs 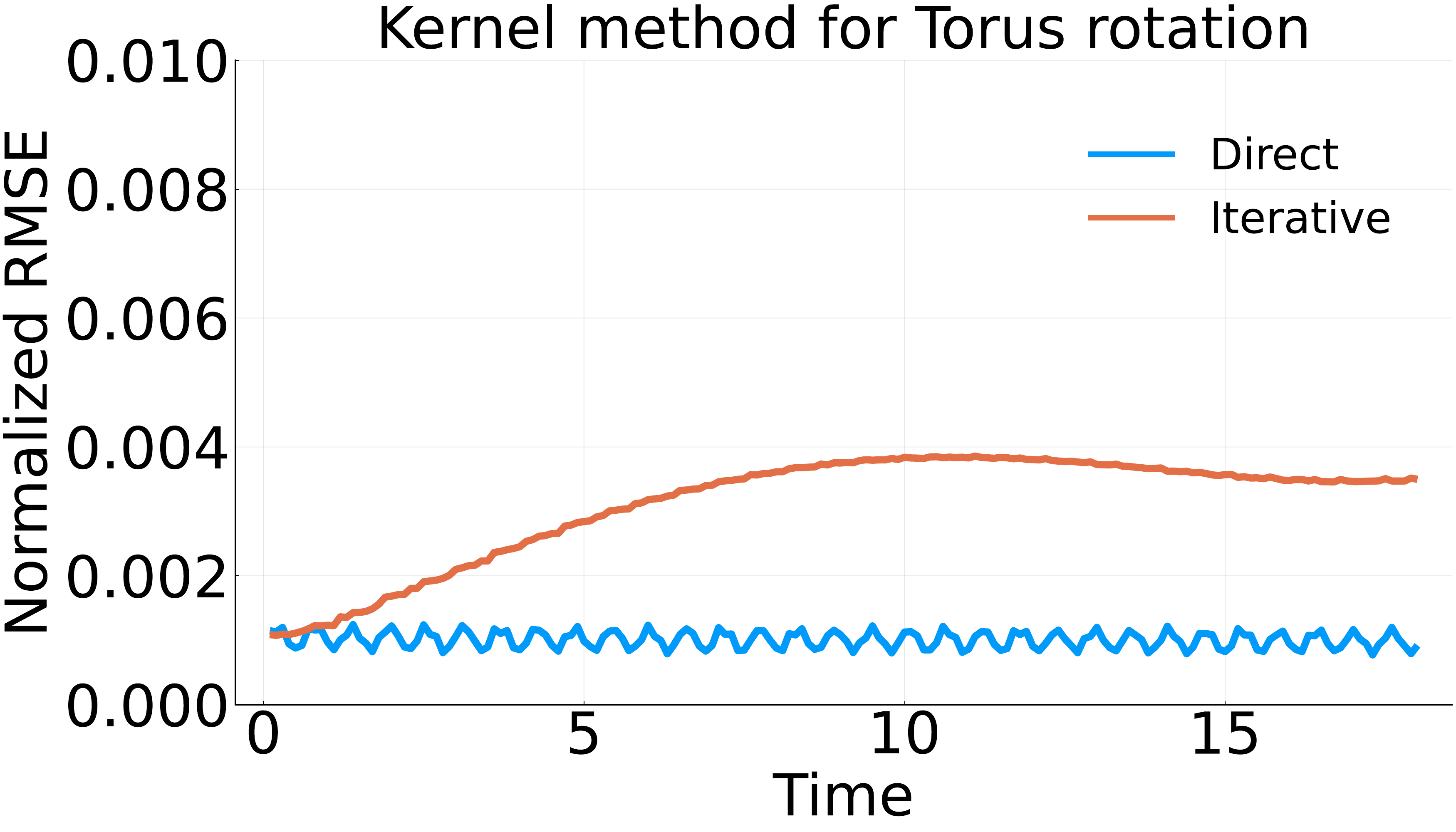}
\includegraphics[width=.48\linewidth]{\figs 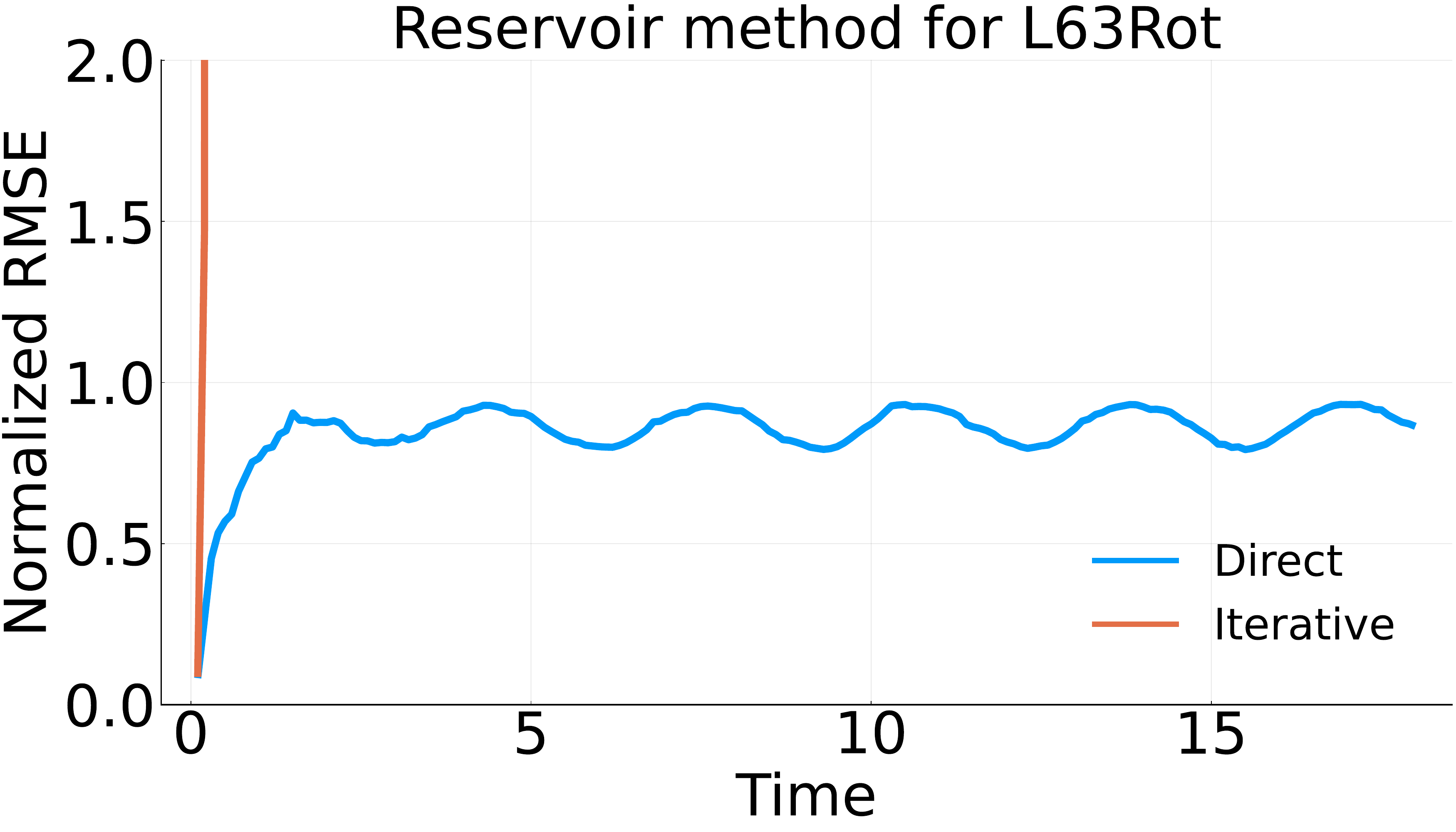}
\includegraphics[width=.48\linewidth]{\figs 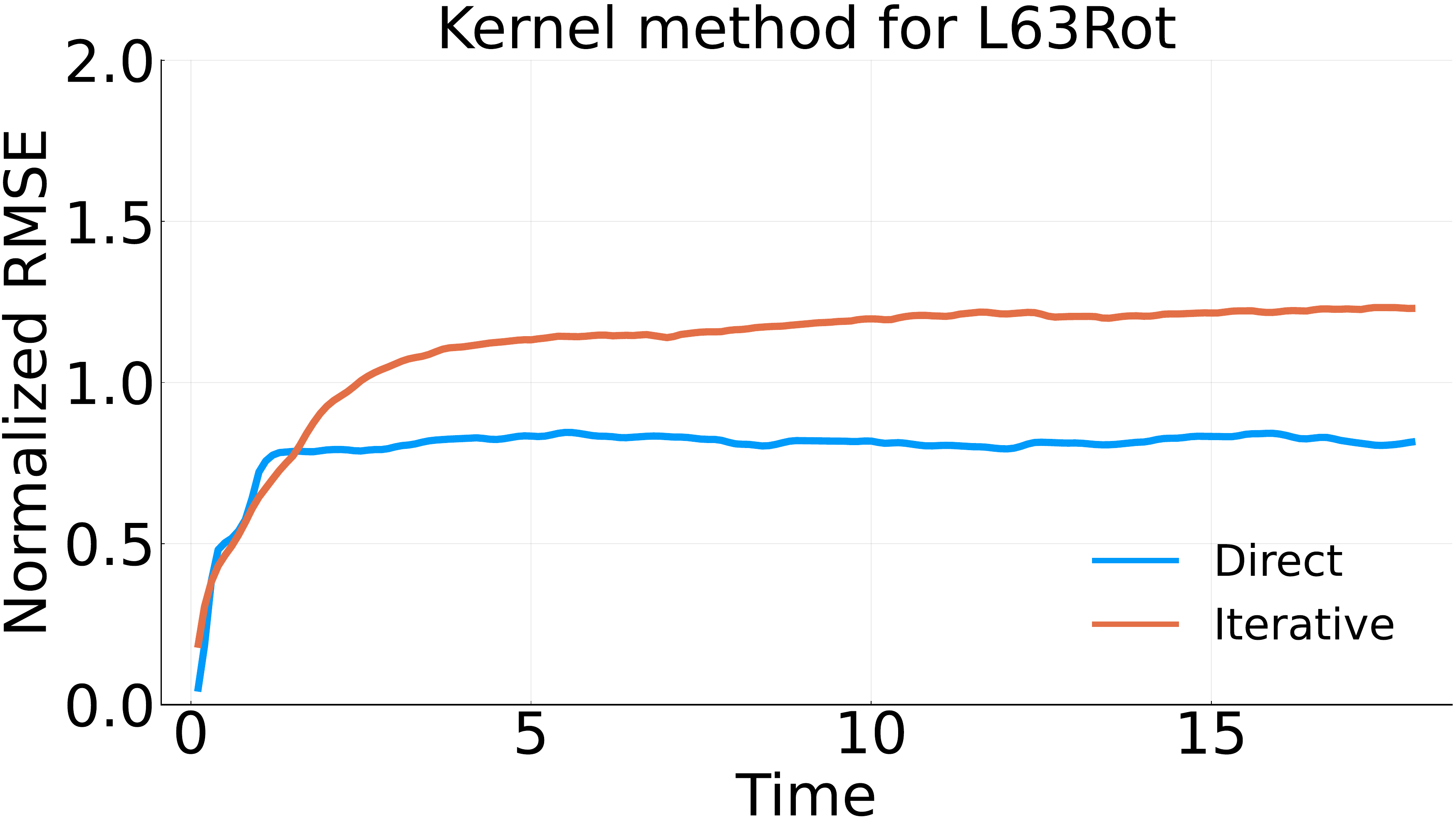}
\caption{Performance of the two reconstruction techniques for (i) a quasiperiodic rotation on the torus $\mathbb{T}^2$ (bottom panels); and (ii) a Cartesian product of L63 with a simple harmonic oscillator (top panels). By Theorem~\ref{thm:direct}~(iv), if one has a proper embedding and a good approximation $\hat{w}$ of $w$ one can achieve arbitrarily small errors for the torus rotation,  for all forecast times. This is supported by the fact that the direct methods for both the paradigms show errors of the order of $10^{-6}$. Since the torus rotation has all Lyapunov exponents zero, by Theorem~\ref{thm:iterative}~(ii) and Theorem~\ref{thm:lambda1}~(ii), the error from the iterative techniques should grow sub-exponentially, as supported by the figures. The system \eqref{eqn:L63Rot} is a mixed spectrum system, i.e., the splitting in \eqref{eqn:def:L2split} is non-trivial. The plots conform to the expected behavior discussed in points (1)--(4) of Section~\ref{sec:discuss}. }
\label{fig:T2_L63Rot}
\end{figure}

We now compare the reconstruction technique using the two paradigms of invariant-graphs and delay-coordinate embedding. The former was implemented using reservoir systems, and the latter using kernel regression. We applied both of these techniques to three systems
\begin{enumerate}[(i)]
    \item A quasiperiodic rotation on a two-dimensional torus [Figure~\ref{fig:T2_L63Rot}, top panel].
    \begin{equation} \label{eqn:T2}
        \left( \theta^1(n+1), \theta^{(2)}(n+1) \right) = \left( \theta^1(n+1), \theta^{(2)}(n+1) \right) + ( 
        \rho_1, \rho_2 ) \bmod 2\pi .
    \end{equation}
    Here $\theta^{(1)}$ and $\theta^{(2)}$ are angular coordinates on the torus, and $( \rho_1, \rho_2 ) $ is the rotation vector.
    \item The Lorenz-63 (L63) system [Figure~\ref{fig:compare}]. Let $\Phi^{t}_{\rm L63}$ denote the flow under the Lorenz63 system. Fix a sampling interval  $\Delta t$. This leads to the discrete time system :
    \begin{equation} \label{eqn:L63}
        \left( x_{n+1}, y_{n+1}, z_{n+1} \right) = \Psi^{\Delta t}_{\rm L63} \left( x_{n}, y_{n}, z_{n} \right) .
    \end{equation}
    $\Phi^{t}_{\rm L63}$ has a unique physical measure which has been proved to be nonuniformly hyperbolic and mixing. 
    \item A dynamical system formed by taking the Cartesian product of L63 with a simple harmonic oscillator, [Figure~\ref{fig:T2_L63Rot}, bottom panel]. Such a system will have a mixed spectrum, with the space $\Disc$ generated by a single base eigenfunction. We shall refer to this system as L63Rot.
    \begin{equation} \label{eqn:L63Rot} \begin{split}
        \theta_{n+1} &= \theta_n + \rho \bmod 2\pi \\
        \left( x_{n+1}, y_{n+1}, z_{n+1} \right) &= \Psi^{\Delta t}_{\rm L63} \left( x_{n}, y_{n}, z_{n} \right) .
    \end{split} \end{equation}
    This system is analysed in the bottom panel of Figure~\ref{fig:T2_L63Rot}.
\end{enumerate}
The results of our computations in Figures \ref{fig:compare}, \ref{fig:T2_L63Rot}, and \ref{fig:error_analysis} illustrate the consequences of Theorem~\ref{thm:direct}, Theorem~\ref{thm:iterative}. Figure~\ref{fig:error_analysis} highlights the two most important conclusions from our results. Firstly, as seen in the top row, the iterative errors grow at an exponential rate comparable to the top Lyapunov exponent $\lambda_1$. Secondly, if the hypothesis space is large enough, then the direct error is bounded above by a formula \eqref{eqn:phi_hypo_autocorr} involving the autocorrelation function of the direct observation map $\phi$. There are three things to note concerning Figure~\ref{fig:error_analysis}:
\begin{enumerate}[(i)]
    \item  Theorem~\ref{thm:iterative} gives an upper bound for the long term behavior of the iterative error. The theoretical bound for exponential rate of growth is indicated by the slope of the black dashed line, and is $\approx 0.9056\Delta t$. So although the initial exponential rate of errors seem to be larger than this, by choosing a multiplicative constant large enough, the error graph still remains underneath the theoretical curve. The offset of the straight dashed line equals the logarithm of this multiplicative constant. Thus as long as the long-term averaged error growth rate is less than $\approx 0.9056\Delta t$, there will always be a multiplicative constant large enough to satisfy the bounds in \eqref{eqn:def:iter_bound} and \eqref{eqn:def:iter_bound_L2}.
    \item The errors from the direct error occasionally cross the theoretical bound indicated by the black dashed line. This is because, the bound in \eqref{eqn:phi_hypo_autocorr} assumes that the learning error for $w$ is zero, i.e., $w$ lies in the hypothesis space $\mathcal{H}$. In most situations such as in our experiments, there is always a small learning error. An extended analysis for this situation is an interesting and open task.
    \item The errors from the iterative forecasts made using the reservoir blow up. This is because the standard reservoir computers are not guaranteed to be stable. This drawback is a important subject for further study.
\end{enumerate}

\begin{figure}[!ht]\center
\includegraphics[width=.48\linewidth]{\figs 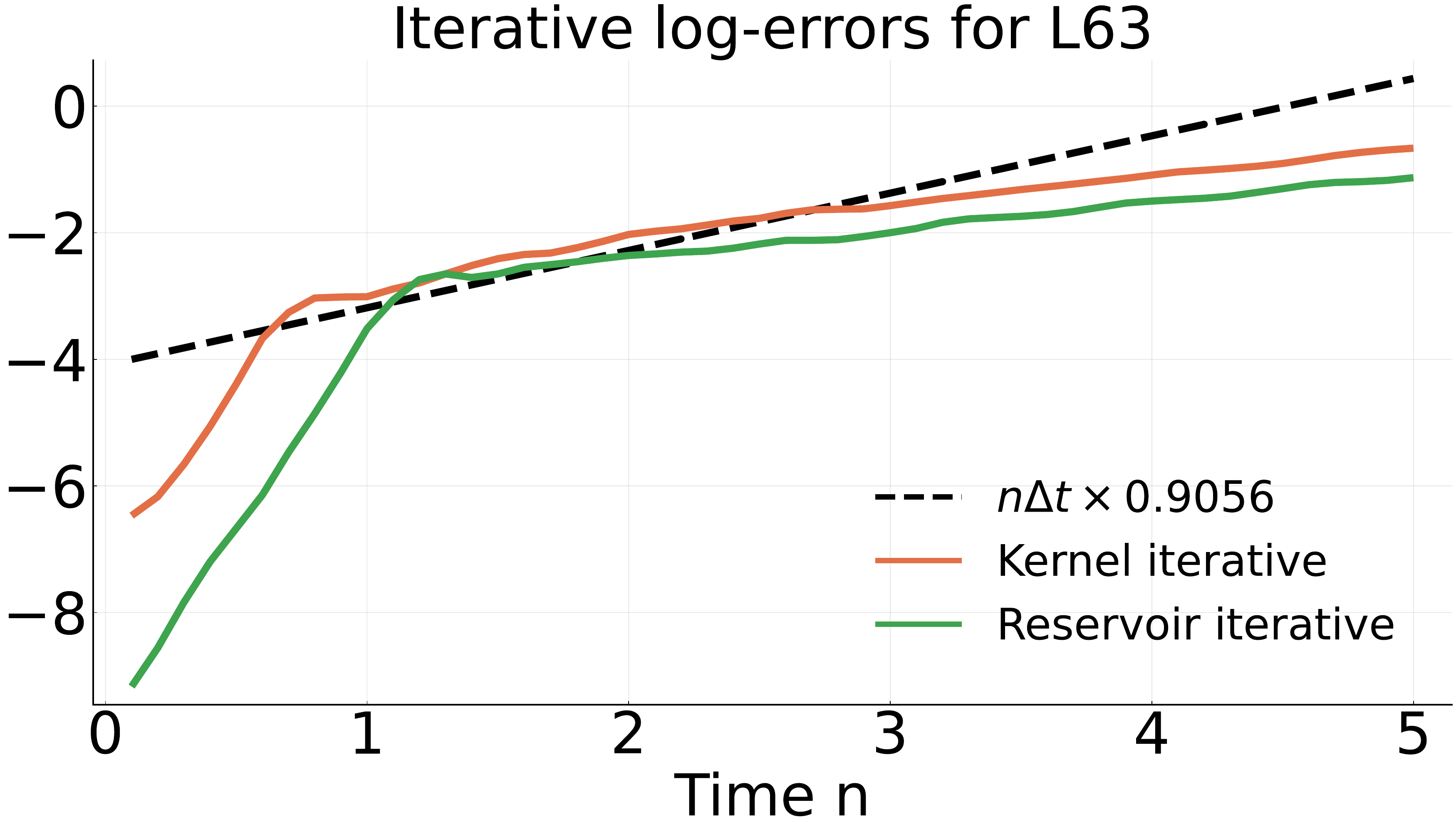}
\includegraphics[width=.48\linewidth]{\figs 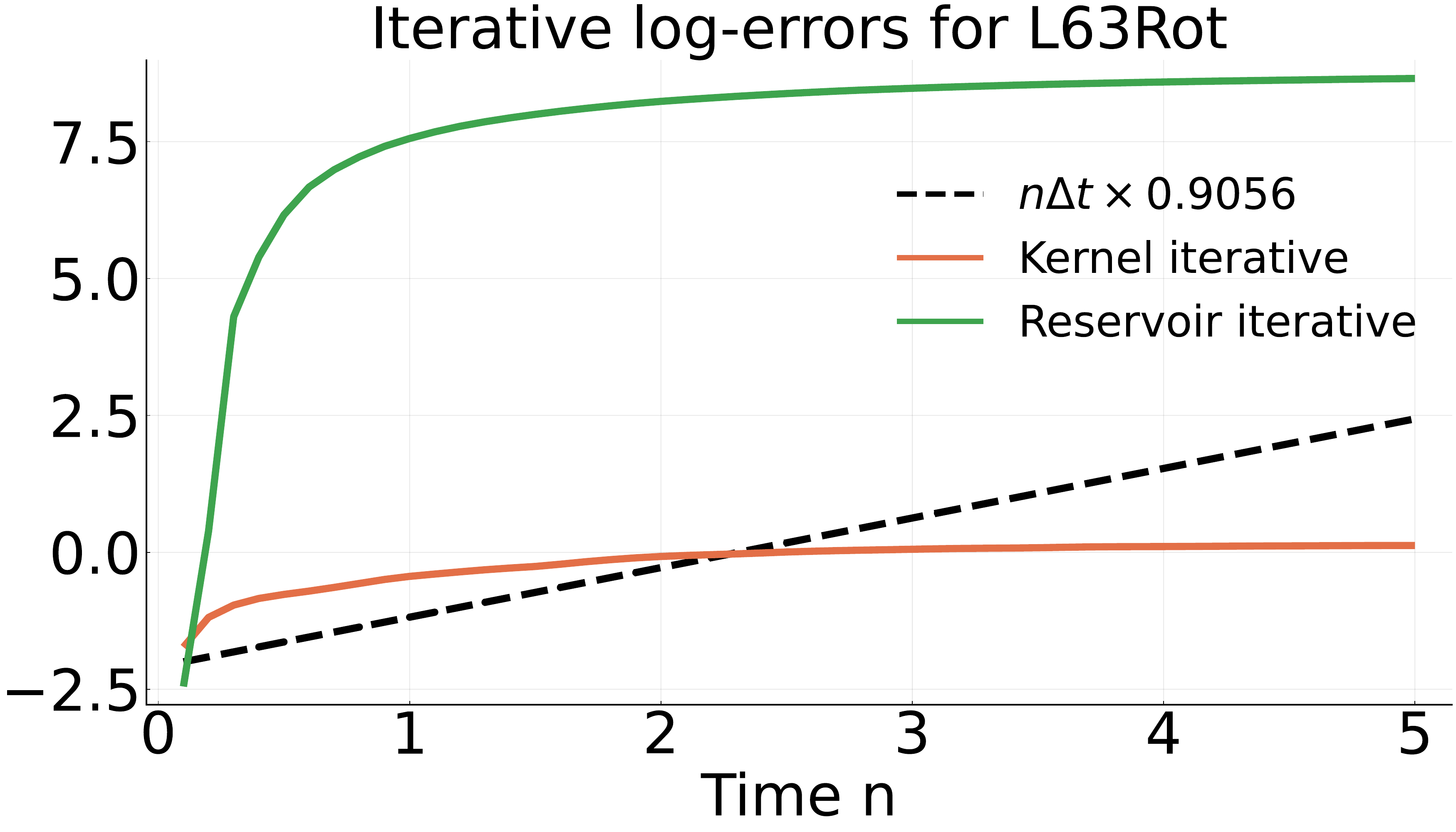}

\includegraphics[width=.48\linewidth]{\figs 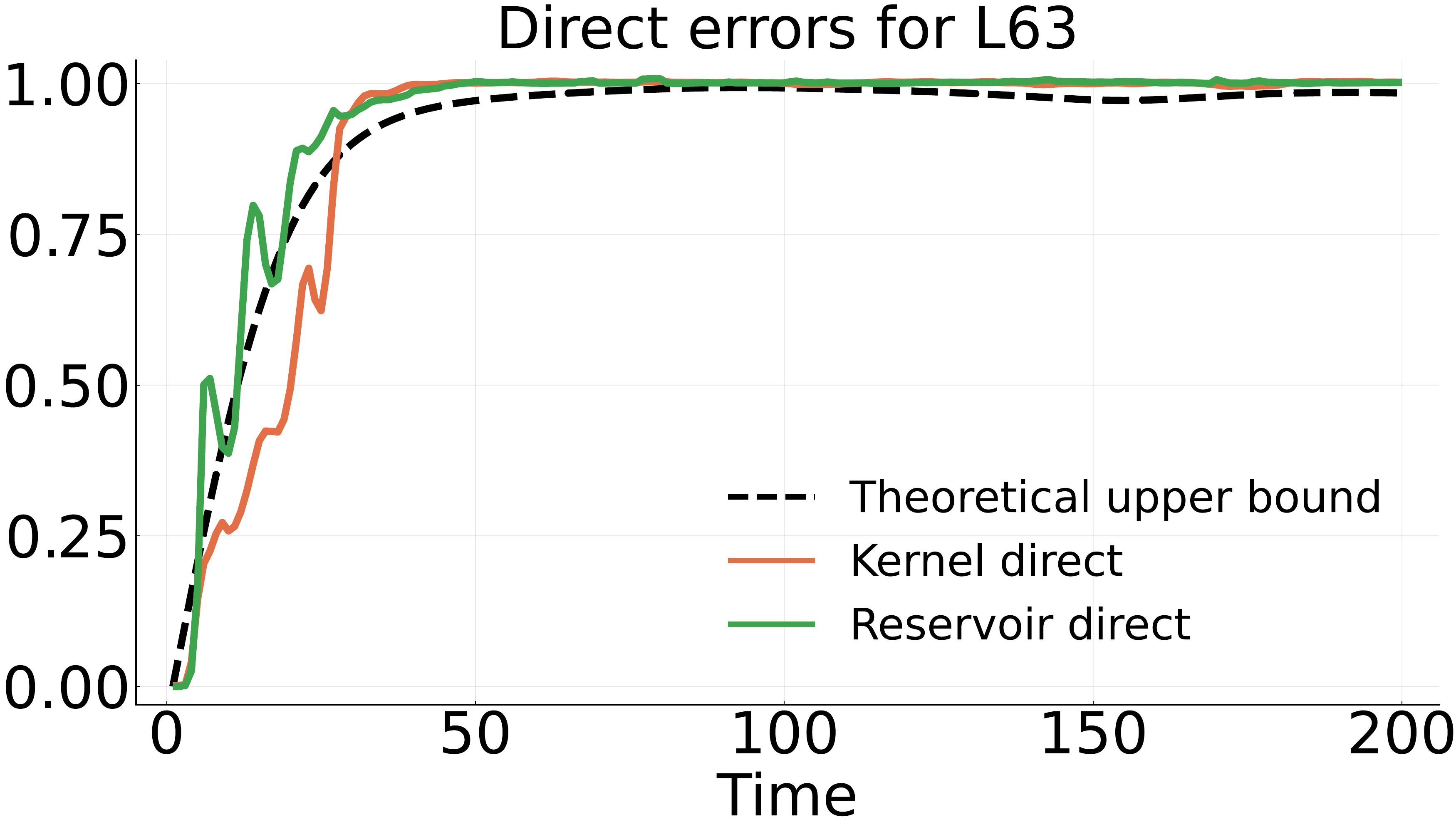}
\includegraphics[width=.48\linewidth]{\figs 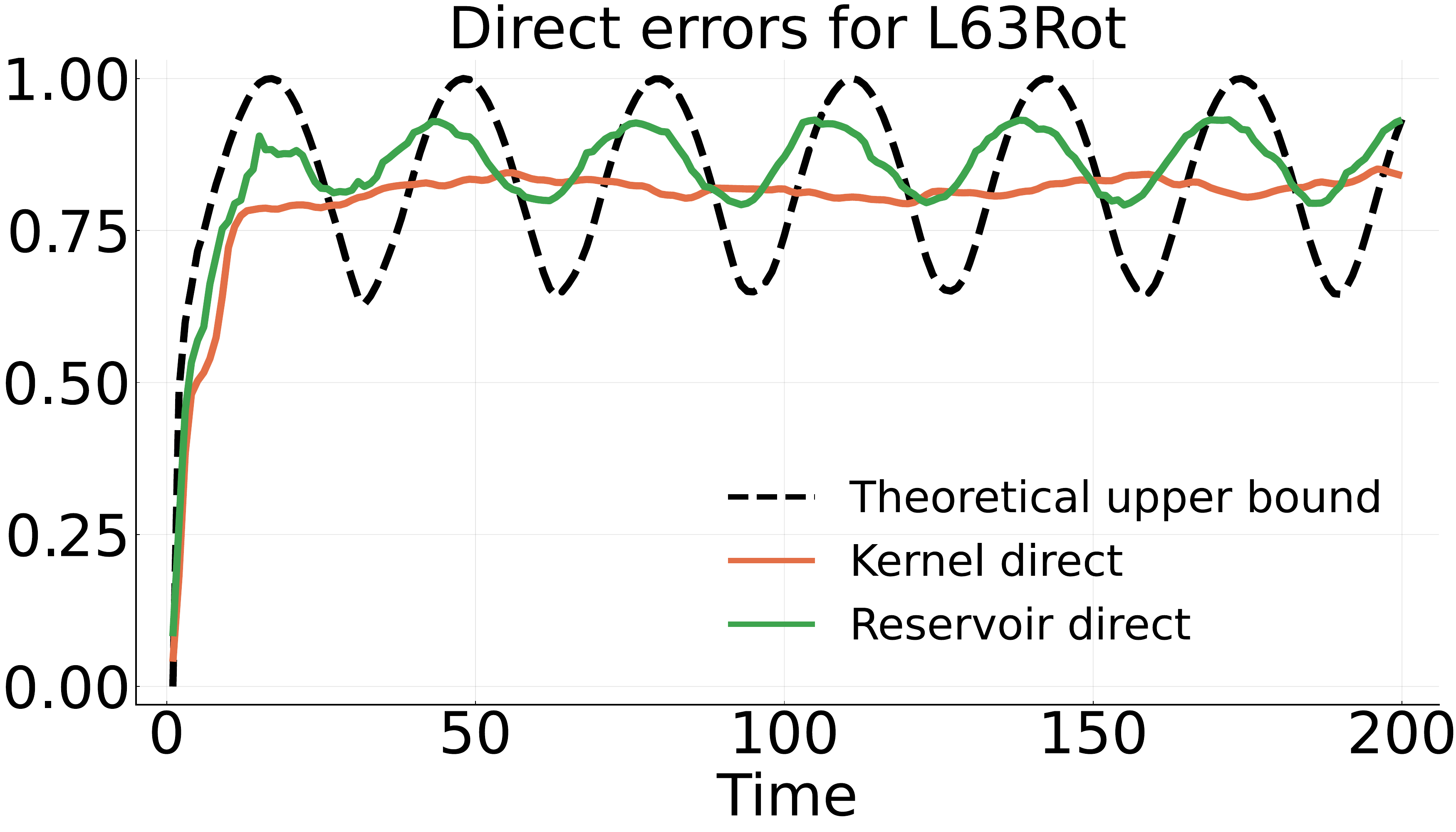}
\caption{ Error analysis using theoretical results. The top row show how the iterative errors of the L63 and L63Rot systems compare with the theoretical bounds of Theorem~\ref{thm:iterative}. The dashed lines each have slopes $\lambda_1 \Delta t$, with $\Delta t$ being the sampling interval, and $\lambda_1 \approx 0.9056$ for both the systems. The bottom row compares the errors from direct forecast with the autocorrelation bound of \eqref{eqn:phi_hypo_autocorr}. The assumption there that $w$ lies in the hypothesis space is not met, and thus we see some fluctuations above the theoretical upper bound. Together, these plots indicate conformity with our theoretical predictions. }
\label{fig:error_analysis}
\end{figure}

This completes the presentation of our main theoretical and numerical results. The framework that we have build provides many new directions of research into the field of learning of dynamical systems. We now present some other directions of work.

%-_-_-_-_-_-_-_-_-_-_-_-_-_-_-_-_-_-_-_-_-_-_-_-_-_-_-_-_-_-_-_-_-_-_-_-_-_-_-_-_-_-_-_-_-_-_-_-_-_-_-_-_-_-_-_-_-_-_-_-_-_-_-_-_-_-_-_-_-_-_-_-_-_-_-_-
\subsection{Methods based on Koopman approximation} \label{sec:Koop_approx}

There are many techniques of forecasting which do not attempt to reconstruct the dynamics using some form of embedding. Instead, they directly try to approximate the Koopman operator, by tracking its action on a limited subspace of functions. In this section we review some of these methods, and relate them loosely to our main mathematical constructions.

%-_-_-_-_-_-_-_-_-_-_-_-_-_-_-_-_-_-_-_-_-_-_-_-_-_-_-_-_-_-_-_-_-_-_-_-_-_-_-_-_-_-_-_-_-_-_-_-_-_-_-_-_-_-_-_-_-_-_-_-_-_-_-_-_-_-_-_-_-_-_-_-_-_-_-_-
\paragraph{Kernel analog forecasting } This technique \cite{ZhaoGiannakis2016} is a direct method for pointwise forecast, using locally decaying kernels. Suppose that $\mu$ is a smooth volume measure, and $p_\epsilon$ is a $C^2$, strictly positive definite, locally decaying kernel, which is Markov wrt $\mu$, i.e.,
\[ \int p_\epsilon( \cdot, y) d\mu(y) \equiv 1_{\Omega}. \]
Let $P_\epsilon$ be the kernel integral operator corresponding to $p_\epsilon$ and $\mu$. The core idea of this technique is the following pointwise estimate \citep[e.g.][]{CoifmanLafon2006, TrillosSlepcev2018}
\begin{equation} \label{eqn:skn4}
\abs{ \phi(\omega) - (P_\epsilon \phi)( \omega) } \leq \norm{ D\phi(\omega) } \epsilon + \bigO{\epsilon^2} , \quad \forall \phi\in C^1(\Omega), \quad \forall \omega \in \Omega .
\end{equation}
Using a change of variables formula, one can write
\[\begin{split}
(P_\epsilon U^t\phi)( \omega) &= \int p_\epsilon( \omega, \omega') (U^t \phi)( \omega' ) d\mu(\omega') = \int p_\epsilon( \omega, \omega') \phi( \Phi^t \omega' ) d\mu(\omega'), \quad \mbox{take } \omega'' := \Phi^t \omega' \\
&= \int p_\epsilon( \omega, \Phi^{-t} \omega'') \phi(\omega'') d (\Phi^t_*\mu)(\omega'') .
\end{split}\]
Therefore
\begin{equation} \label{eqn:dskn3}
( U^t\phi )(\omega) = \int p_\epsilon( \omega, \Phi^{-t} \omega'') \phi(\omega'') d (\Phi^t_*\mu)(\omega'') + \norm{ D (U^t\phi)(\omega) } \bigO{\epsilon} .
\end{equation}
In the above inequality the integral is approximated as
\[ \int p_\epsilon( \omega, \Phi^{-t} \omega'') \phi(\omega'') d (\Phi^t_*\mu)(\omega'') \approx \frac{1}{N} \sum_{n=0}^{N-1} p_\epsilon \left( \omega, \omega_{n} \right) \phi( \omega_{n+t} ) . \]
One of the major drawbacks of this method is that the pointwise approximation deteriorates as $t$ increases, since the function $U^t \phi$ becomes increasingly oscillatory.

%-_-_-_-_-_-_-_-_-_-_-_-_-_-_-_-_-_-_-_-_-_-_-_-_-_-_-_-_-_-_-_-_-_-_-_-_-_-_-_-_-_-_-_-_-_-_-_-_-_-_-_-_-_-_-_-_-_-_-_-_-_-_-_-_-_-_-_-_-_-_-_-_-_-_-_-
\paragraph{Diffusion forecast} This involves choosing an orthonormal basis $\SetDef{ \phi_j }{ j \in \num }$ for some choice of a Hilbert space $H$, choosing the size of a truncation $L$, setting $H_L := \spn \SetDef{ \phi_j }{ j = 1,\ldots, L }$, and setting
\[ U^{(L)} := \pi_L U \pi_L, \quad U^{(L)}_n := \pi_L U^n \pi_L , \]
where $\pi_L : H\to H_L$ is the orthogonal projection. $U^{(L)}$ and $U^{(L)}_n$ are this $L$-dimensional approximations of the Koopman operator. Typical choices of $H$ are $L^2(\mu)$ or Sobolev spaces, and the $\phi_j$ are typically Laplacian eigenfunctions, or eigenfunctions of symmetric kernel integral operators. The choice between $U^{(L)}$ and $U^{(L)}_n$ is similar to the choice between the iterative and direct methods \eqref{eqn:def:err_iter} and \eqref{eqn:def:err_direct}. However, since these methods are not dependent on an actual embedding of the dynamics, the error of both of these forecasts grow at the same rate as the rate of decay of correlations.

%-_-_-_-_-_-_-_-_-_-_-_-_-_-_-_-_-_-_-_-_-_-_-_-_-_-_-_-_-_-_-_-_-_-_-_-_-_-_-_-_-_-_-_-_-_-_-_-_-_-_-_-_-_-_-_-_-_-_-_-_-_-_-_-_-_-_-_-_-_-_-_-_-_-_-_-
\paragraph{Spectral techniques} The diffusion forecast in one among many techniques of approximating the Koopman operator. A more robust approach is a spectral approximation technique developed in \cite{DGJ_compactV_2018}, in which the goal is to approximate the spectral measure of the generator $V$ associated with a continuous time dynamical system. This technique is convergent and works for any kind of ergodic dynamical system. The Koopman group $\{ U^t : t\in\real \}$ is then approximated by the 1-parameter unitary group generated by a compact, spectral approximation $\tilde{V}$ of $V$. In this technique, $U^t$ is not approximated by its action on a fixed subspace of functions, but on a subspace spanned by \emph{approximate eigenfunctions}. This also leads to a discovery of nearly-periodic structures present within the possibly chaotic system.

%-_-_-_-_-_-_-_-_-_-_-_-_-_-_-_-_-_-_-_-_-_-_-_-_-_-_-_-_-_-_-_-_-_-_-_-_-_-_-_-_-_-_-_-_-_-_-_-_-_-_-_-_-_-_-_-_-_-_-_-_-_-_-_-_-_-_-_-_-_-_-_-_-_-_-_-
\subsection{Future work} There are several promising directions of research that can be built upon our framework.
\begin{enumerate}
\item Multi-modal forecasting : One of the main ideas verified theoretically and via numerical experiments is that the error of direct forecasts increase at the rate of mixing of the system, which is usually larger than the top Lyapunov exponent. However, if there are quasiperiodic components, the direct method is effective in retaining that component. On the other hand, the error from the iterative method increases at a slower rate, but does not preserve the quasiperiodic components of the signal. The iteration model eventually behaves effectively in an uncorrelated fashion with the true dynamics. A \emph{multi-modal} forecasting technique would be a combination of these two modes, which combines their best features. 
\item The direct method is essentially $\pi U^n \phi$ and the iterative is $U^{n-1}\pi U\phi$. Another possibility is a $k$-time step iterative forecast which would be $U^{n-k}\pi U^k \phi$. As $k$ increases the leading term $\pi U^k \phi$ will have an error that decays according to the decay of correlations. To make amneds, we could incorporate several $k$-s in a window $[1,K]$ as
\[ \sum_{k=1}^K \alpha_k U^{n-k}\pi U^k \phi. \]
\item Yet another idea we are pursuing is \emph{ensemble forecasting} which has long been suggested as a prediction technique for chaotic systems \cite{HammelEtAl1988, DanforthYorke2006}.
\item Numerical approximation of optimal feedback function $\bar{w}$ - a key insight of Theorem~\ref{thm:lambda1} is that the top Lyapunov exponent of the reconstructed model depends on the behavior of the feedback function in a neighborhood of the image of the attractor. Theorem ~\ref{thm:iterative} then shows that this Lyapunov exponent describes the exponential rate at which the reconstructed model diverges from the true system. In most learning techniques, one tries to find a feedback function that simultaneously minimizes a fitting error, and a an oscillation penalty term. Theorems \ref{thm:lambda1} and \ref{thm:iterative} suggest that instead of measuring the overfitting error via the usual oscillation bond, a good candidate would be to take into account the behavior of $\bar{w}$ in a neighborhood of the dataset. The precise manner in which this ambient space behavior is to be translated into a penalty function is a promising field of research. A related and inseparable question of an appropriate choice of hypothesis space.
\item Effect of noise : Our techniques have not addressed the challenges posed by noise, either in measurement or dynamic. It is well known \citep[e.g.][]{SugiharaMay90} that measurement noise could be hard to distinguish from chaos and can severely restrict the accuracy of even short term predictions. Numerical averages rely on ergodic convergences, and the stability of ergodic averages to noise is a complicated and broad question of its own. Stability results have been shown in systems with SRB measures \cite{CowiesonYoung2005, BlumenthalYoung_equiv_2019}. In such settings, the use of Kalman filtering in a model free approach  \cite{HamiltonEtAl2016, HamiltonEtAl2017} may yield promising results.

\end{enumerate}

%-_-_-_-_-_-_-_-_-_-_-_-_-_-_-_-_-_-_-_-_-_-_-_-_-_-_-_-_-_-_-_-_-_-_-_-_-_-_-_-_-_-_-_-_-_-_-_-_-_-_-_-_-_-_-_-_-_-_-_-_-_-_-_-_-_-_-_-_-_-_-_-_-_-_-_-
\section{Review of non-uniform hyperbolicity} \label{sec:cocyc}

This section provides an overview of the topics of matrix cocycles and Lyapunov exponent theory.

%-_-_-_-_-_-_-_-_-_-_-_-_-_-_-_-_-_-_-_-_-_-_-_-_-_-_-_-_-_-_-_-_-_-_-_-_-_-_-_-_-_-_-_-_-_-_-_-_-_-_-_-_-_-_-_-_-_-_-_-_-_-_-_-_-_-_-_-_-_-_-_-_-_-_-_-
\subsection{Matrix cocycles} \label{sec:cocyc:basics}

Let Assumption~\ref{A:f} hold, and $G : \Omega\to GL(\real, m)$ be a measurable map. Then it generates a \emph{matrix cocycle} (see \cite{FroylandEtAl_coherent_2010}, \citep[][3.4]{Arnold_random_1991}), which is the map
\begin{equation}\label{eqn:ih7b}
\cocyc:\Omega\times \num_0\to GL(\real;m), \quad \cocyc(n, \omega) := 
\begin{cases}
\Id_d &\mbox{if } n=0\\
G(f^{n-1} \omega ) \cdots G(\omega) &\mbox{if } n>0 \\
G( f^{-|n|} \omega)^{-1} \cdots G(f^{-1} \omega )^{-1} &\mbox{if } n<0
\end{cases} .
\end{equation}
$\cocyc$ is called a $GL(m;\real)$-valued cocycle over the dynamics $(\Omega,\mu,f)$ generated by $f$. It has the property
\begin{equation}\label{eqn:def:cocyc}
\cocyc(m+n, \omega) = \cocyc(n, f^{m} \omega) \cdot \cocyc(m, \omega), \quad \forall \omega\in \Omega, \quad \forall m,n\in \integer. 
\end{equation}
Here the $\cdot$ notation denotes the matrix multiplication. Equation \eqref{eqn:def:cocyc} is the defining equation of a matrix cocycle. Conversely, given any map $\mathcal{G}:\Omega\times\num_0\to GL(m;\real)$ satisfying \eqref{eqn:def:cocyc}, one has a generator $G : \Omega \to GL(m;\real)$ so that $\cocyc$ is related to $G$ via \eqref{eqn:ih7b}. One of the immediate consequences of \eqref{eqn:def:cocyc} is that
\[ \cocyc(0, \omega) = \Id_m, \quad \cocyc( -n, \omega ) = \cocyc( n, f^{-n} \omega )^{-1}, \quad \forall \omega\in \Omega . \]
When the initial point $\omega_0\in\Omega$ is fixed, we will drop it from the notation and define
\[ \cocyc(n-1,j) := \cocyc(n-j, f^{j} \omega_0) = G(f^{n-1} \omega_0) \cdots G(f^{j} \omega_0) . \]
Matrix valued cocycles arise naturally in multiple ways in dynamical systems. For example, if $\Omega$ is a $m$-dimensional manifold and $f$ a differentiable map, then $\cocyc( \omega,n) := Df^n(\omega)$ is a $GL(m;\real)$ cocycle. 

\begin{proposition}[Multiplicative ergodic theorem ] \label{prop:MET} \cite{Ruelle_Lyapu_1979} \citep[][Thm 4.1, pg 10]{FroylandEtAl_coherent_2010}
Let Assumption~\ref{A:f} hold, then there exists a forward invariant set $\Omega'$ of full $\mu$-measure such that the limit
\[\Lambda(\omega) := \lim_{n\to\infty} \Matrix{ \cocyc(n, \omega)^* \cocyc(n, \omega) }^{1/2n}\]
exists for every $\omega\in \Omega'$. Moreover, there is a splitting $\real^M = \oplus_{i=1}^{l} E_i(\omega)$ and constants $\lambda_1 \geq \ldots \geq \lambda_l \geq -\infty$ such that
\[v\in E_i \imply \lim_{n\to\infty} \frac{1}{n} \ln \norm{ \cocyc(n, \omega) v } = \lambda_i .\]
The numbers $l$, $\lambda_1, \ldots, \lambda_l$ are constant on $\Omega'$. The subspaces $E_i$ depend measurably on $\omega\in \Omega'$ and 
\begin{equation} \label{eqn:oixks4}
    \cocyc(n, \omega) v \in E_i\left( f^n \omega \right), \quad \forall v\in E_i(\omega), \, \forall n\in \integer .
\end{equation}
\end{proposition}

The vector spaces $E_i$ are called the Lyapunov subspaces and $\lambda_i$ the Lyapunov exponents. These measure the asymptotic rate of expansion or contraction along the Lyapunov directions.

%-_-_-_-_-_-_-_-_-_-_-_-_-_-_-_-_-_-_-_-_-_-_-_-_-_-_-_-_-_-_-_-_-_-_-_-_-_-_-_-_-_-_-_-_-_-_-_-_-_-_-_-_-_-_-_-_-_-_-_-_-_-_-_-_-_-_-_-_-_-_-_-_-_-_-_-
\subsection{Pesin sets} \label{sec:cocyc:pesin} 

Lyapunov exponents describe the asymptotic behavior of orbits and not the local differential properties of the map. By Proposition~\ref{prop:MET}, at almost every $\omega$, the limits $\lambda_i$ are attained along the various Oseledet subspaces $E_i$. However, the rate at which the limits are attained are in general not uniform or even continuous as a function of $\omega$. The Oseledet subspace of $T_\omega \tilde{\Omega}$ corresponding to $\lambda_i$ itself is usually a measurable but non-continuous function of $\omega$.

Pesin sets \cite{Pesin_families_1976, Ruelle_Lyapu_1979} were introduced to capture the regularity and boundedness in the highly non-uniform nature of the Oseledet splitting. Fix a constant $\leak>0$, such that $\leak<< \min_i \abs{ \lambda_i }$. $\leak$ is called the \emph{leakage rate}. Then there is a nested sequence of compact sets $\Omega_1\subseteq \Omega_2 \subseteq \Omega_3 \subseteq \ldots$ whose union has $\mu$-measure $1$, such that for every $k\in\num$
\[ e^{-k\leak} e^{(\lambda_i-\leak)n} \leq \norm{Df^n(\omega) | E_i} \leq e^{k\leak} e^{(\lambda_i+\leak)n} \quad \forall \omega\in \Omega_k, \, \forall i\in \{1,\ldots,r\}, \, \forall n\in\integer . \]
Moreover, the subspaces $E_i$ vary smoothly on the sets $\Omega_k$.

Although the norm of the Jacobian $Df^n(\omega)$ when restricted to $E_i$ grows asymptotically at the rate $e^{\lambda_i n}$, this exact exponential growth need not be attained for finite $n$. There is a constant $\mult(\omega) = \mult(\omega; \leak)$ depending on $\omega$ such that
\begin{equation} \label{eqn:NUH_fluct}
\frac{1}{\mult(\omega)} e^{(\lambda_i-\leak)n} \leq \norm{Df^n(\omega) | E_i} \leq \mult(\omega) e^{(\lambda_i+\leak)n} , \quad \forall n\in\integer, \; \forall i\in 1,\ldots, r .
\end{equation}
$\mult(\omega)$ plays the role of a multiplicative constant, and $\leak$ behaves as the extent of fluctuation around the limiting rate $\lambda_i$. The decomposition into Pesin sets imply that if $\omega$ is restricted to $\Omega_k$, then $\mult(\omega)$ can be uniformly bounded by $e^{k\leak}$.

Thus the Pesin sets $\Omega_k$ have uniformly hyperbolic behavior, However they need not be uniformly hyperbolic sets, as they are not necessarily invariant sets. In general $f(\Omega_k) \subseteq \Omega_{k+1}$. Note that if $\Omega_k$ is invariant, then it is a uniformly hyperbolic set. In spite of not being invariant sets, Pesin sets are useful for obtaining concrete bounds on the rate of hyperbolicity. The Poincare recurrence theorem guarantees that a typical trajectory returns to a Pesin set infinitely many times. These two properties of recurrence and uniform hyperbolic rates has been used effectively to establish strong global properties of the system, such as approximation by periodic points \cite{Katok_periodiC_1980}, shadowing \cite{WangSun2010} and metric properties of local stable and unstable manifolds \cite{LedYoung_SRBI_1985, LedYoung_SRBII_1985}. These techniques will play an important role in our proofs.

\paragraph{$L^p$ Pesin sets} For every choice of the leakage rate $\leak$, the Pesin sets $\Omega_k$ grow to form an invariant set of full measure. The rate at which $\mu(\Omega_k)$ approaches $1$ is an important consideration. It is important for obtaining estimates on various statistical properties of the nonuniformly hyerbolic system. However, there is not much estimates on how quickly the Pesin sets grow, except under additional conditions, such as the existence of reasonably good Markov approximations \citep[e.g.][]{GouezStoy2019}. For our purpose, we  say that a nonuniformly hyperbolic system has \emph{$L^p$ Pesin sets} if the function $\omega\mapsto \mult(\omega; \leak)$ is $L^p$-integrable wrt $\omega$. This property will be used by us to obtain global bounds from local behavior, in Theorem~\ref{thm:GraphT_growth} later.

%-_-_-_-_-_-_-_-_-_-_-_-_-_-_-_-_-_-_-_-_-_-_-_-_-_-_-_-_-_-_-_-_-_-_-_-_-_-_-_-_-_-_-_-_-_-_-_-_-_-_-_-_-_-_-_-_-_-_-_-_-_-_-_-_-_-_-_-_-_-_-_-_-_-_-_-
\subsection{Lyapunov exponents in Euclidean space} \label{sec:cocyc:metric} 

Given a dynamical system on $F:\real^M \to \real^M$, one has the following alternative definition of Lyapunov exponents :
\[\lambda(z,v) := \limsup_{n\to\infty} \frac{1}{n} \ln \lim_{\delta\to 0^+} \frac{1}{\delta} \norm{ F^n(z+\delta v) - F^n(z) }, \quad z\in \real^M, \, v\in\real^M. \]

\begin{lemma} \label{lem:sjdfb4}
Let $\reconstruct : \real^M \to \real^M$ be a $C^1$ map, with an invariant ergodic measure $\bar{\mu}$ with compact support $X$. Let $T_{X} \real^M = E_1\oplus E_2\oplus \cdots\oplus E_k$ be a splitting of $T\real^M$ restricted to $X$. Let $\lambda_1$ be the maximal Lyapunov exponent wrt the measure $\bar{\mu}$. Then for $\mu$-a.e. $z\in X$,
\[ \lambda_1 = \max_{1\leq i\leq k} \sup_{v\in E_i(z)\setminus\{0\}} \lambda(z,v) = \sup_{v\in T_z \real^M \setminus\{0\}} \lambda(z,v) . \]
% = \sup_{v\in E_1(z)\setminus\{0\}} \lambda(z,v)
\end{lemma}

The proof is a direct consequence of the definition of Lyapunov exponents and will be omitted. We next consider a special type of perturbed sequences of points.

\paragraph{Pseudo-trajectories} Let $\reconstruct: \real^M \to \real^M$ be a $C^1$ map on a manifold $M$, with an ergodic invariant measure $\bar{\mu}$ with compact support $X$. Fix a sequence of positive numbers $\left(c_j\right)_{j=0}^\infty$ and initial point $z_0\in X$. Now define
\begin{equation}\begin{split} \label{eqn:def:pseudo}
\mathcal{S} \left( z_0, \left(c_j\right)_{j=0}^\infty \right) &:= \SetDef{ \left(z_j\right)_{j=0}^\infty \in \real^M }{ z_{n+1} = \reconstruct(z'_n), \, d(z'_{n+1}, z_{n+1}) \leq c_n d(z'_n, z_n), \, \forall n\in\num }. \\
\mathcal{S} \left( \delta; z_0, \left(c_j\right)_{j=0}^\infty \right) &:= \SetDef{ \left(z_j\right)_{j=0}^\infty \in \mathcal{S} \left( z_0, \left(c_j\right)_{j=0}^\infty \right) }{ d(z'_0, z_0) \leq \delta }.
\end{split}\end{equation}
Thus $\mathcal{S} \left( z_0, \left(c_j\right)_{j=0}^\infty \right)$  is the set of all pseudo-trajectories $z_n \in \real^m$ such that at each stage $n$, $z_{n+1}$ is the image of a perturbation $z'_n$ of $z_n$. Moreover, the perturbation to $z_{n+1}$ is at most $c_n$ times the perturbation to $z_n$. The set $\mathcal{S} \left( \delta; z_0, \left(c_j\right)_{j=0}^\infty \right)$ are the subset of these sequences such that the initial perturbation is no more than $\delta$. Thus $S(z_0, c) = \cup_{\delta>0} S(\delta; z_0, c) $ The figure below illustrates such a sequence in $\mathcal{S} \left( \delta; z_0, \left(c_j\right)_{j=0}^\infty \right)$ :
\[\begin{tikzcd}
z_0 \arrow[dashed]{d}{+\vec\delta_0} \arrow{r}{\reconstruct} & \reconstruct(z_0)  \arrow{r}{\reconstruct} & \reconstruct^2(z_0)  \arrow{r}{\reconstruct} & \ldots  \arrow{r}{\reconstruct} & \reconstruct^n(z_0) \arrow[dashed]{ddd}{\dev(n,\delta)} \\
z'_0 \arrow{r}{\reconstruct} & z_1 = \reconstruct(z'_0)  \arrow[dashed]{d}{+\vec\delta_1} \\
\ &z'_1 \arrow{r}{\reconstruct} & z_2 = \reconstruct(z'_1) \arrow[dashed]{d}{+\vec\delta_2} \\
\ & \ & z'_2 \arrow{r}{\reconstruct} & \ldots \arrow{r}{\reconstruct} & z_n = \reconstruct(z'_{n-1})
\end{tikzcd}\]
The top row shows a the reference trajectory $\{ \reconstruct^n z_0 : n\in\num_0 \}$, starting at an initial point $z_0$. At each time step $n=1,2,\ldots$, $z_n$ is the image of a point $z'_{n-1}$, which is a $\vec\delta_{n-1}$ perturbation of the earlier point $z_{n-1}$. The magnitude of $\vec\delta_{n}$ is bounded by $c_{n-1} \delta_{n-1}$. Thus at every stage, the error accumulates and is scaled by the factor of at most $c_n$. The magnitude of the initial perturbation is $\norm{\vec\delta_0} \leq \delta$. 
Such perturbed sequences arise in our proof of Theorem~\ref{thm:lambda1}. We study the rate of growth of the divergence between the two trajectories, as a ratio of the initial error magnitude $\delta$. The maximum possible deviation after $n$ steps can be written as
\[\dev( z_0, n, \delta) := \sup \SetDef{  d\left( z_{n}, \tau^n(z_0) \right) }{ \left(z_j\right)_{j=0}^\infty \in \mathcal{S} \left( \delta;z_0, \left(c_j\right)_{j=0}^\infty \right) } . \]
We are more interested in the growth of this deviation as a multiplier of the initial error margin $\delta$, namely
\[ \dev(z_0, n) := \limsup_{\delta\to 0^+} \frac{1}{\delta}\dev( z_0, n, \delta) .\]
%
%and its asymptotic rate of growth with $n$, i.e., %\[ \dev = \dev\left( z_0, \left(c_j\right)_{j=0}^\infty \right) := \limsup_{n\to\infty}  \frac{1}{n} \ln \dev(n). \]
%
We next derive the asymptotic rate at which these rate of divergences $\dev(z_0, n)$ grow.

\begin{proposition}[$\delta$-pseudo trajectory] \label{prop:srg7}
Let $\reconstruct : M\to M$ be a $C^1$ map on a manifold $M$, with an ergodic, non-uniformly hyperbolic invariant measure $\bar{\mu}$ with compact support $X$. Assume the notation in \eqref{eqn:def:pseudo}. Let $\left(c_j\right)_{j=0}^\infty$ be a sequence positive numbers for which the limit $C = \lim_{N\to\infty} \frac{1}{N} \sum_{j=0}^{N} \ln c_j $ exists. Then 
\[ \limsup_{n\to\infty} \frac{1}{n} \ln \dev(z_0, n) \leq \lambda_1(\tau) + C , \quad \bar{\mu}-\mbox{a.e. } z_0\in X . \]
\end{proposition}

The proof is a direct consequence of the local stable / unstable manifold theorem \citep[e.g.][Sec 6]{Ruelle_Lyapu_1979} and will be omitted.
Let $c:X\to \real^+$ be a continuous function. Now define similarly to \eqref{eqn:def:pseudo},
\begin{equation}\begin{split} \label{eqn:def:pseudoII}
\mathcal{S} \left( z_0, c \right) &:= \SetDef{ \left(z_j\right)_{j=0}^\infty }{ z_{n+1} = \reconstruct(z'_n), \, d(z'_{n+1}, z_{n+1}) \leq c(z_n) d(z'_n, z_n), \, \forall n\in\num }. \\
\mathcal{S} \left( \delta; z_0, c \right) &:= \SetDef{ \left(z_j\right)_{j=0}^\infty \in \mathcal{S} \left( z_0, c \right) }{ d(z'_0, z_0) \leq \delta }.\\
\dev(n ,\delta; z_0, c) &:= \sup \SetDef{ d\left( z_n, \reconstruct^n(z_0) \right) }{ \left( z_j \right)_{j=0}^{\infty} \in \mathcal{S} \left( \delta; z_0, c \right) } .
\end{split}\end{equation}

\begin{proposition}[$\delta$-pseudo trajectory II] \label{prop:aedf2}
Let $\reconstruct : M\to M$ be a $C^1$ map on a manifold $M$, with an ergodic, non-uniformly hyperbolic invariant measure $\bar{\mu}$ with compact support $X$, and $c: M\to \real^+$ is a continuous map. Assume the notation in \eqref{eqn:def:pseudoII}. Then
\[ \limsup_{n\to\infty} \frac{1}{n} \limsup_{\delta\to 0^+} \ln \frac{1}{\delta} \dev( n, \delta; z_0, c) \leq \lambda_1(f) + \int \ln c d\bar\mu . \]
\end{proposition}

\begin{proof} Set $c_j := c( z_j )$. Then as $\delta\to 0$, the sum $\frac{1}{n} \sum_{j=0}^{n-1} \ln c_j$ converge by the uniform continuity of $c$ and the ergodic theorem. Thus Proposition~\ref{prop:srg7} applies.
\end{proof}

%-_-_-_-_-_-_-_-_-_-_-_-_-_-_-_-_-_-_-_-_-_-_-_-_-_-_-_-_-_-_-_-_-_-_-_-_-_-_-_-_-_-_-_-_-_-_-_-_-_-_-_-_-_-_-_-_-_-_-_-_-_-_-_-_-_-_-_-_-_-_-_-_-_-_-_-
\section{Cocycles with random perturbations} \label{sec:cocyc:peturb}

This section investigates a type of dynamics that we shall call a \emph{perturbed random cocycle} in \eqref{eqn:ab_exact}. The results in this section are of independent interest, but also directly apply to our study of the growth of error under iterations of the reconstructed system \eqref{eqn:def:feedback_1}. We shall later use them in the proof to Theorem~\ref{thm:iterative}. A version of perturbed random cocycles were investigated by Barreira and Valls \cite{BarreiraValls2006, BarreiraValls2005} in the context of non-autonomous differential equations of the form
\[ \frac{d}{dt} v(t) = A(t) v(t) + f_{\text{perturb}} \left( v(t), t \right) ,  \]
in Euclidean space. The non-autonomous behavior is due to the dependent on the time parameter $t$. In our case, for a fixed initial state $\omega_0$ of the underlying dynamical system $(\Omega, f)$, the time dependence is  via the orbit of $\omega_0$. The authors assumed that the function $f_{\text{perturb}}$ decay at a polynomial rate wrt the norm of $v$, a property not applicable to our case. We shall study the problem in more generality.

Consider dynamics of the form
\begin{equation} \label{eqn:def:cocyc_perturb_n}
z_{n+1} = G_n z_n + d_{n+1} ,
\end{equation}
where the $z_n$s and $d_n$s are $m$-vectors, and $G_n$s are invertible $m\times m$ matrices. $z_n$ represents a state vector. The $G_n, d_n$ form a sequence of random matrices and perturbation vectors. If we assume that the source of this randomness is the dynamical system $(\Omega, \mu, f)$, then the iterates of \eqref{eqn:def:cocyc_perturb_n} can be realized as iterates of the skew product map
\begin{equation} \label{eqn:def:cocyc_perturb}
F : \Omega \times \real^M \to \Omega \times \real^M, \quad F : \left(\begin{array}{c} \omega \\ z \end{array} \right) \mapsto \left(\begin{array}{c} f\omega \\ G(\omega) z + d(f\omega) \end{array} \right) ,
\end{equation}
where $d\in L^2\left( \Omega, \mu ;\real^M\right)$ can be interpreted as a random perturbation vector. We call the system \eqref{eqn:def:cocyc_perturb} a \emph{perturbed random matrix cocycle}. If we fix an initial state $(\omega_0, z_0)$ and setting
\[ z_n := \proj_2 F^n(\omega_0, z_0), \quad G_n := G(f^n\omega_0), \quad d_n := d(f^n \omega_0), \quad n\in\num_0 . \]
then we get \eqref{eqn:def:cocyc_perturb_n}.
Iterating \eqref{eqn:def:cocyc_perturb_n} $n$ times gives
\begin{equation} \label{eqn:cx38f}
z_{n} = \left[ G_{n-1} \cdots G_{0} \right] z_0 + d_n + \sum_{j=1}^{n-1} \left[ G_{n-1} \cdots G_{j} \right] d_j = \cocyc(n-1,0) z_0 + \sum_{j=1}^{n} \cocyc(n-1,j) d_j . %, \quad \cocyc(n,n+1) := \Id, \, \cocyc(n,j) := G_n \cdots G_j .
\end{equation}
Equation \eqref{eqn:cx38f} can be rewritten in terms of \eqref{eqn:def:cocyc_perturb} as
\begin{equation}\label{eqn:Fn}
F^n \left( \omega, z \right) = \left( f^n \omega, \, \cocyc(n,\omega) z + \left(\GrowthT^n d \right)(\omega) \right) .
\end{equation}
where
\begin{equation} \label{eqn:def:GraphT}
\left(\GrowthT^n d \right)(\omega) := \sum_{j=1}^{n} \cocyc \left( n-j, f^{j} \omega \right) d\left( f^{j} \omega \right) \in \real^d ,
\end{equation}
for every $n\in \num$. We shall call this map $\GrowthT^n$ the \emph{graph-transform} operator. If one views the map $d$ as an $\real^d$-valued graph over $\Omega$, then $\GrowthT^n d$ is a new graph over $\Omega$. Also note that the transformation is linear in $d$, justifying the name of ``operator". Moreover,
\begin{equation} \label{eqn:kjbc93}
\GrowthT^0 u \equiv 0, \quad \GrowthT^1 u \equiv u, \quad \forall u : \Omega\to \real^d .
\end{equation}
Moreover, if the initial value $z_0=0$, then 
\[ z_n(\omega) = \proj_2 F^n(\omega, z_0\equiv 0) = (\GrowthT^n d)(\omega), \quad \forall n\in\num . \]
By \eqref{eqn:Fn}, the growth of $z_n$ depends on the initial value $z_0$ only through the action of $\cocyc(n,\omega)$, which is well tractable by the multiplicative ergodic theorem. We are mainly interested in the remaining part, i.e., the behavior of the operator $\GrowthT^n$. In the following theorem, we shall use $\lambda_i^+$ to denote $\max( \lambda_i, 0 )$.

\begin{theorem} \label{thm:GraphT_growth}
Let Assumption~\ref{A:f} hold and suppose $\cocyc$ is a $GL(m;\real)$-valued cocycle as in \eqref{eqn:ih7b}, \eqref{eqn:def:cocyc}; along with a perturbed matrix cocycle as in \eqref{eqn:def:cocyc_perturb}. Assume the notations of the associated Oseledet splitting as in Proposition~\ref{prop:MET}, so that the vector valued function $d$ has the splitting
\[ d = \oplus_{i=1}^{r} d^{(i)} , \quad d^{(i)}\in E^{(i)} . \]
Finally let $\GrowthT$ be as in \eqref{eqn:def:GraphT}. Then
\begin{enumerate}[(i)]
\item Suppose $d$ is essentially bounded. Then for $\mu$-a.e. $\omega\in \Omega$,
\begin{equation} \label{eqn:GraphT_growth_1}
\norm{\GrowthT^n d^{(i)}(\omega) } = \norm{d^{(i)}}_{L^\infty} \mult(\omega) \bigO{ e^{ (\lambda^+_i + \leak) n} }, \mbox{ as } n\to\infty, \quad \forall 1\leq i\leq r .
\end{equation}
where $\lambda_i^+ := \max(\lambda_i, 0)$.
\item If the system has $L^2$ Pesin sets, then
\begin{equation} \label{eqn:GraphT_growth_3}
\norm{\GrowthT^n d^{(i)}(\omega) }_{L^2(\mu)} = \norm{d^{(i)}}_{L^2(\mu)} \norm{\mult}_{L^2(\mu)}^2 \bigO{ e^{ (\lambda^+_i + \leak) n} }, \mbox{ as } n\to\infty, \quad \forall 1\leq i\leq r .
\end{equation}
%
%\item In particular, if $z_0 = 0$ and $d:\Omega\to \real^M$ is a continuous function, then % \begin{equation} \label{eqn:GraphT_growth_2} \norm{z_n(\omega)} = \norm{\GrowthT^n d (\omega) } = \norm{d}_{C(\Omega)} \bigO{ e^{(\lambda^+_i + \leak) n}} .\end{equation}
%
\end{enumerate}
\end{theorem}

\begin{proof} %By \eqref{eqn:def:GraphT}, we have for every $m,n\in\num$, %\begin{equation}\label{eqn:GraphT_mn} \left(\GrowthT^{m+n} d \right)(\omega) = \cocyc \left( m, f^n \omega \right) \left(\GrowthT^{n} d \right)(\omega) + \left(\GrowthT^{m} d \right) \left( f^n \omega \right) = \cocyc \left( n, f^m \omega \right) \left(\GrowthT^{m} d \right)(\omega) + \left(\GrowthT^{n} d \right) \left( f^m \omega \right) . \end{equation}
To gain more insight into the growth of $z_n$, we use \eqref{eqn:def:cocyc} to get
\[\begin{split}
\left(\GrowthT^{n} d \right)(\omega) &= \sum_{j=1}^{n} \cocyc \left( n-j, f^{j} \omega \right) d\left( f^{j} \omega \right) = \sum_{j=1}^{n} \cocyc \left( n-j, f^j \omega \right) \cocyc \left( j, \omega \right) \cocyc \left( j, \omega \right)^{-1} d\left( f^j \omega \right) \\
&= \cocyc \left( n, \omega \right) \sum_{j=1}^{n} \cocyc \left( j, \omega \right)^{-1} d\left( f^j \omega \right) .
\end{split}\]
We have thus related the error in the prediction to the growth of the vector $\cocyc \left( n, \omega \right)$. The growth of the matrix $\cocyc \left( n, \omega \right)$ with $n$ can be estimated using Proposition~\ref{prop:MET}. Broadly speaking, the different Oseledet subspaces grow approximately at rate $e^{\lambda_i n}$ under the action of $\cocyc \left( n, \omega \right)$. We now estimate the growth of the components of the summand along the Oseledet splitting. Define 
\[ e_j(\omega) := \cocyc \left( j, \omega \right)^{-1} d\left( f^{j} \omega \right), \quad e_j^{(i)} \left( \omega \right) := \cocyc \left( j, \omega \right)^{-1} d^{(i)}\left( f^{j} \omega \right). \] 
Then we have
\begin{equation} \label{eqn:srn8}
\cocyc \left( n-j, f^{j} \omega \right) d\left( f^{j} \omega \right) = \cocyc \left( n, \omega \right) e_j(\omega), \quad \cocyc \left( n-j, f^{j} \omega \right) d^{(i)}\left( f^{j} \omega \right) = \cocyc \left( n, \omega \right) e^{(i)}_j(\omega) .
\end{equation}
The analysis of the growth of this term will depend on the sign of $\lambda_i$. 

\paragraph{Case : $\lambda_i>0$} Then
\[ \norm{ e_j^{(i)}(\omega) } = \norm{ \cocyc(j, \omega)^{-1} d^{(i)} (f^j \omega) } \leq \norm{ \cocyc(j, \omega)^{-1} \vert_{ E^{(i)}(f^j \omega)} } \norm{ d^{(i)} (f^j \omega) } \leq \mult(\omega) e^{-j(\lambda_i-\epsilon)} \norm{ d^{(i)} (f^j \omega) } \]
Therefore by \eqref{eqn:NUH_fluct} and \eqref{eqn:srn8},
\[\begin{split}
\norm{\cocyc \left( n-j, f^{j} \omega \right) d^{(i)}\left( f^{j} \omega \right)} & = \norm{ \cocyc \left( n, \omega \right) e^{(i)}_j(\omega) } \leq \norm{\cocyc \left( n, \omega \right) \vert_{ E^{(i)}(\omega) } } \norm{ e^{(i)}_j(\omega) } \\
& \leq \mult^2(\omega) e^{n(\lambda_i+\epsilon)} e^{-j(\lambda_i-\epsilon)} \norm{ d^{(i)} \left( f^j \omega \right) } . % \\&= \mult(\omega)^2 e^{} \norm{ d^{(i)} \left( f^j \omega \right) } .
\end{split}\]
Thus $\lambda_i>0$ implies
\begin{equation} \label{eqn:ind33p}
\norm{ \left( \GrowthT^n d^{(i)} \right) (\omega) } \leq \sum_{j=1}^{n} \norm{\cocyc \left( n-j, f^{j} \omega \right) d^{(i)}\left( f^{j} \omega \right)} \leq \mult(\omega)^2 e^{n(\lambda_i+\epsilon)} \sum_{j=1}^{n} e^{-j(\lambda_i-\epsilon)} \norm{ d^{(i)} \left( f^j \omega \right) } .
\end{equation}
At this point the following identity is relevant to us :
\begin{equation} \label{eqn:scjn3}
e^{n(\lambda_i+\epsilon)} \sum_{j=1}^{n} e^{-j(\lambda_i-\epsilon)} = e^{-(\lambda_i-\epsilon)} e^{n(\lambda_i+\epsilon)} \frac{ 1- e^{-n(\lambda_i-\epsilon)} }{ 1- e^{-(\lambda_i-\epsilon)} } = \frac{e^{-(\lambda_i-\epsilon)}}{ 1- e^{-(\lambda_i-\epsilon)} } \left[ e^{n(\lambda_i+\epsilon)}  - e^{2\epsilon} \right] = c_i \bigO{ e^{n(\lambda_i+\epsilon)} },
\end{equation}
for some constant $c_i>0$. Suppose that $d$ is essentially bounded. Then \eqref{eqn:ind33p} and \eqref{eqn:scjn3} gives
\[ \lambda_i>0 \imply  \norm{ \left( \GrowthT^n d^{(i)} \right) (\omega) } \leq c_i \norm{ d^{(i)} }_{L^\infty} \bigO{ e^{n(\lambda_i+\epsilon)} } . \]
This proves Claim~(i) for the case $\lambda_i>0$. Now suppose that the map has $L^2$ Pesin sets. Integrating both sides of  \eqref{eqn:ind33p} with respect to $\omega$ and then summing according to \eqref{eqn:scjn3} gives
\begin{equation} \label{eqn:njn48}
\norm{\GrowthT^n d^{(i)}}_{L^1(\mu)} = \int \norm{ \left( \GrowthT^n d^{(i)} \right) (\omega) } d\mu(\omega) \leq C^2 \norm{ d^{(i)}}_{L^1(\mu)} \bigO{e^{n(\lambda_i+\epsilon)}} .
\end{equation}

\paragraph{Case : $\lambda_i<0$} Suppose that $\omega\in \Omega_k$. In this case note that  for each $1\leq j\leq n$,
\[ \norm{ \cocyc\left( n-j, f^j \omega \right) d^{(i)}\left( f^j \omega \right) } \leq \norm{ \cocyc\left( n-j, f^j \omega \right) \vert_{ E^{(i)}(f^j \omega) } } \norm{ d^{(i)}\left( f^j \omega \right) } \leq e^{ (n-j)(\lambda_i+\leak)} e^{\leak (k+j)} \norm{ d^{(i)}\left( f^j \omega \right) } .\]
Summing over $j$ gives
\begin{equation} \label{eqn:pe93} \begin{split}
\lambda_i<0 \imply \norm{ \left( \GrowthT^n d^{(i)} \right) (\omega) } \leq e^{k\leak} e^{n(\lambda_i+\leak)} \sum_{j=1}^{n} e^{ j \lambda_1 } \norm{ d^{(i)}\left( f^j \omega \right) }
\end{split}\end{equation}
The rest of the `analysis is similar to the previous analysis, now made more simpler by the fact that $\lambda_i<0$ and the RHS above is bounded uniformly with respect to $n$.
This completes the proof of theorem.
\end{proof}

%-_-_-_-_-_-_-_-_-_-_-_-_-_-_-_-_-_-_-_-_-_-_-_-_-_-_-_-_-_-_-_-_-_-_-_-_-_-_-_-_-_-_-_-_-_-_-_-_-_-_-_-_-_-_-_-_-_-_-_-_-_-_-_-_-_-_-_-_-_-_-_-_-_-_-_-
\section{Proof of Proposition~\ref{thm:kncd3l}} \label{sec:proof:kncd3l}

Proposition~\ref{thm:kncd3l} is a direct consequence of a result of J. Stark, which we state below.

\paragraph{Skew-product systems} Let $\tilde{\Omega}, Y$ be smooth manifolds and $T:(x,y) \mapsto (fx,g(x,y))$ be a skew-product map on $\tilde{\Omega}\times Y$. For every $n\in\num$, let $g^{(n)}: \tilde{\Omega}\times Y\to Y$ be the map such that
\[ T^n(x,y) = \left( f^n x, g^{(n)}(x,y) \right), \quad \forall (x,y) \in \tilde{\Omega}\times Y . \]

\begin{lemma}[Invariant graphs for skew-product systems] \label{lem:Stark}
\citep[][Thm 1.3]{Stark1999} Assume the notations above, and suppose that the following hold :
\begin{enumerate}[(i)]
    \item $f$ is a $C^{1+\alpha}$ diffeomorphism and thre are constants $\mu\geq 0, C_2>0$ such that $\| Df^{-n} \| \leq C_2 e^{\mu n}$.
    \item There is a closed and $f$-invariant subset $\Omega\subseteq \tilde{\Omega}$.
    \item There exist constants $\lambda, C_3>0$ such that
    \begin{equation} \label{eqn:pd83}
    \Lip\left( g^{(n)}(x,\cdot) \right) \leq C_3\exp(-\lambda n), \quad \forall x\in \tilde{\Omega} .
    \end{equation}
    \item $g$ is uniformly $C^{1+\alpha}$ on compact sets.
\end{enumerate}
Then there is a continuous map  $\Phi:\Omega\to Y$ such that the graph of $\Phi$ is invariant and globally attracting under $T$. Moreover, for every $\gamma\in (0,\alpha]$ such that $\mu(1+\gamma)<\lambda$, $\Phi$ is $C^{1+\gamma}$ in the Whitney sense.
\end{lemma}

Note that for every $n\in\num$, $x\in \tilde{\Omega}, y\in Y$,
\[ g^{(1)} = g, \quad g^{(n+1)}(x,y) = g\left( f^n (\omega), g^{(n)}(x,y) \right). 
\]
For our purposes, set $\tilde{\Omega}=\Omega$ and $Y=\real^L$, and let $f,g$ be smooth maps satisfying Assumptions \ref{A:f} and \ref{A:pPhi} and \eqref{eqn:cv93m}. Then clearly all the conditions of Lemma~\ref{lem:Stark} are satisfied. Thus there is a smooth map $\Phi:\Phi\to \real^L$ whose graph is invariant under $T$. This completes the proof of theorem. \qed

%-_-_-_-_-_-_-_-_-_-_-_-_-_-_-_-_-_-_-_-_-_-_-_-_-_-_-_-_-_-_-_-_-_-_-_-_-_-_-_-_-_-_-_-_-_-_-_-_-_-_-_-_-_-_-_-_-_-_-_-_-_-_-_-_-_-_-_-_-_-_-_-_-_-_-_-
\section{Proof of Theorem~\ref{thm:lambda1}} \label{sec:proof:lambda1}

%-_-_-_-_-_-_-_-_-_-_-_-_-_-_-_-_-_-_-_-_-_-_-_-_-_-_-_-_-_-_-_-_-_-_-_-_-_-_-_-_-_-_-_-_-_-_-_-_-_-_-_-_-_-_-_-_-_-_-_-_-_-_-_-_-_-_-_-_-_-_-_-_-_-_-_-
\subsection{Proof of Claims (i), (ii) }

Claim~(i) was proved by Dechert and Gencay. We restate their result using our terminology. Alhough they prove their result in the context of delay coordinate maps, their proof is based on a commutation identity \citep[see][eqn 3.4]{DechertGencay1996} which also holds in our more general case.

\begin{lemma} \label{lem:Lyap_sub} \citep[][Thm 3.1]{DechertGencay1996} 
Let $M, N$ be $C^1$ manifolds of dimension $m, n$ respectively. Let $f:M\to M$ and $g:N\to N$ be two $C^1$ diffeomorphisms, conjugate via a $C^1$ map $J:M\to N$ as $g\circ J = J \circ f$. Let $\mu$ be an invariant ergodic measure $\mu$ of $f$. Let $\lambda_1(f,\mu) > \cdots >\lambda_r(f,\mu)$ be the distinct Lyapunov exponents of the ergodic system $(f,\mu)$. Let $E_1\oplus \cdots \oplus E_r$ be the corresponding Oseledet splitting.
\begin{enumerate} [(i)]
    \item For every $1\leq j \leq r$, $\lambda_j = \lambda_j(f,\mu)$ is also a Lyapunov exponent of the ergodic system $(g, J_*\mu)$. The Oseledet subspace of $TN$ corresponding to $\lambda_j$ contains the subspace $DJ(E_j)$.
    \item In particular, the Lyapunov exponents of $g$ contains as a subset the Lyapunov exponents of $f$.
\end{enumerate}
\end{lemma}

In our case, the conjugation is via the map
\[ h:= (\phi, \Phi) : \Omega \to \real^{d+L}. \]
We next prove Claim~(ii) Under our assumption of ergodicity of $\mu$, the Lyapunov exponents are constant $\mu$-a.e. and coincide with their averages. The semi-continuity of averaged Lyapunov exponents is well known, both as functions of the map \citep[e.g.][Prop 2.2.]{BochiViana_conti_2005}; or as a function of a cocycle over a fixed base dynamics \citep[e.g.][Rem 1.4]{Viana_Lyap_2020}. 

%-_-_-_-_-_-_-_-_-_-_-_-_-_-_-_-_-_-_-_-_-_-_-_-_-_-_-_-_-_-_-_-_-_-_-_-_-_-_-_-_-_-_-_-_-_-_-_-_-_-_-_-_-_-_-_-_-_-_-_-_-_-_-_-_-_-_-_-_-_-_-_-_-_-_-_-
\subsection{Proof of Claim~(iii)}

Fix a generic point $\omega_0\in\support(\mu)$ and set 
\[ z_0 := \left( \phi(\omega_0), \Phi(\omega_0) \right) = h(\omega_0) .\]
$z_0$ is a point in $X$. To determine the maximal Lyapunov exponent of $\reconstruct$, we have to determine the maximum rate of deviation of orbits under perturbations. By Lemma~\ref{lem:sjdfb4}, it is sufficient to consider the perturbation to occur either only in the first $d$ coordinates or last $L$ coordinates in the space $\real^{d+L}$. We call these $\phi$-perturbations and $\Phi$-perturbations respectively.

\paragraph{$\phi$-perturbations} First perturb $z_0$ to $z'_0 = \left( \phi(\omega_0) + \vec\delta, \Phi(\omega_0) \right)$ for some $\vec\delta\in\real^d$. Then
\[ \reconstruct(z'_0) = \reconstruct\left( \phi(\omega_0) + \delta, \Phi(\omega_0) \right) = \left( w\circ\Phi(\omega_0), g\left( \phi(\omega_0) + \delta, \Phi(\omega_0) \right) \right) . \]
Therefore setting $\delta' = g\left( \phi(\omega_0) + \delta, \Phi(\omega_0) \right) - g\left( \phi(\omega_0), \Phi(\omega_0) \right)$, we get
\begin{equation} \label{eqn:skfn34}
\reconstruct\left( z_0 + \left(\begin{array}{c} \delta \\ 0 \end{array}\right) \right) = \reconstruct \left( z_0 \right) + \left(\begin{array}{c} 0 \\ \delta' \end{array}\right), \quad \norm{\delta'} \leq \norm{\partial_1 g}_{\sup} \delta . 
\end{equation}
Thus by Assumption~\ref{A:g_contract}, the map is contractive under $\phi$-perturbations. In light of this observation, it is sufficient to bound the rate of growth of $\Phi$-perturbations by $\lambda_1(f)$.

\paragraph{$\Phi$-perturbations} Next perturb $z_0$ to $z'_0 = \left( \phi(\omega_0), \Phi(\omega_0) + \vec\delta_0 \right)$ for some $\vec\delta_0\in \real^L$. Then
\[ \reconstruct(z'_0) = \reconstruct \left(\begin{array}{c} \phi(\omega_0) \\ \Phi(\omega_0) + \vec\delta_0 \end{array}\right) = \left(\begin{array}{c} \hat{w}\left( \Phi(\omega_0) + \vec\delta_0 \right) \\ g\left( \phi(\omega_0), \Phi(\omega_0)+ \vec\delta_0 \right) \end{array}\right) . \]
We wish to show that if one starts with a $\Phi$-perturbation of $z_0$, then after one iteration of $\reconstruct$, one still ends up with a $\Phi$-perturbation with the image of a perturbed point. More precisely, we have the following picture,
\begin{equation} \label{eqn:fdn38} \begin{tikzcd}
\blue{z_0} \arrow[blue]{r}{+\vec\epsilon_0} \arrow[red]{d}{ + \left(0,\vec{\delta}_0 \right) } & \blue{ z_0'' } \arrow[blue]{r}{\reconstruct} & \blue{z_1} \arrow[red]{d}{ + \left(0,\vec{\delta}_1 \right) } \\
z_0' \arrow{rr}{\reconstruct} & & \reconstruct z'_0
\end{tikzcd} \end{equation}
as described below.

\begin{lemma} \label{lem:sdm83}
Let Assumptions \ref{A:f}, \ref{A:pPhi}, \ref{A:g_contract} and \ref{A:w_ext} hold. Let $z_0 = h(\omega_0)\in X$ and $z'_0$ be a $\Phi$-perturbation of $z_0$ by a vector $\vec\delta\in \real^L$. Then there is a point $z_1\in X$ such that
\begin{enumerate} [(i)]
    \item $z_1 = \reconstruct(z''_0)$ for some point $z''_0\in X$ such that the perturbation $\vec \epsilon_0 := z''_0 - z_0$ has length at most $\Csensr( \omega_0 ) \norm{\vec\delta_0}$ from $z_0$, where
    \[ \Csensr( \omega ) := \left(1 + \Csens(\omega) \right) \kappa_{\retract}, \quad \forall \omega\in \Omega . \]
    \item $\reconstruct(z'_0)$ is a $\Phi$-perturbation of $z_1 = \reconstruct(z''_0)$ with perturbation magnitude at most $\left[ 1 + \Csensr(\omega_0) \right] \norm{\vec\delta_0}$.
\end{enumerate}
\end{lemma} 

Before proving Lemma~\ref{lem:sdm83} we show how its repeated application leads to a proof of Theorem~\ref{thm:lambda1}~(iii). Repeated applications of \eqref{eqn:fdn38} gives
\begin{equation} \label{eqn:8iep}
\begin{tikzcd}
\blue{z_0} \arrow[blue]{r}{+\vec\epsilon_0} \arrow[red]{d}{ + \left(0,\vec{\delta}_0 \right) } & \blue{ z_0'' } \arrow[blue]{r}{\reconstruct} & \blue{z_1} \arrow[red]{d}{ + \left(0,\vec{\delta}_1 \right) }  \arrow[blue]{r}{+\vec\epsilon_1} & \blue{ z_1'' } \arrow[blue]{r}{\reconstruct} & \blue{z_2} \arrow[red]{d}{ + \left(0,\vec{\delta}_2 \right) }  \arrow[dotted, blue]{r}{} & \blue{z_n} \arrow[red]{d}{ + \left(0,\vec{\delta}_n \right) } \arrow[blue]{r}{+\vec\epsilon_n} & \blue{z_n''} \arrow[blue]{r}{\reconstruct} & \blue{z_{n+1}} \arrow[red]{d}{ + \left(0,\vec{\delta}_{n+1} \right) }\arrow[blue, dotted]{r}{} & \blue{\cdots} \\
z_0' \arrow{rr}{\reconstruct} & & \reconstruct z'_0 \arrow{rr}{\reconstruct} & & \reconstruct^2 z'_0  \arrow[dotted]{r}{} & \reconstruct^n z'_0 \arrow{rr}{\reconstruct} & & \reconstruct^{n+1} z'_0 \arrow[dotted]{r}{} &\cdots 
\end{tikzcd}
\end{equation}
Thus we have

\begin{lemma} \label{lem:sefjn4}
Let Assumptions \ref{A:f}, \ref{A:pPhi}, \ref{A:g_contract} and \ref{A:w_ext} hold. Then for every $\omega_0\in \Omega$ and any $\Phi$-perturbation $z'_0$ of $z_0 = h(\omega_0)$, there is a sequence of points $z''_0, z''_1, z''_2 \ldots \in X$ such that for every $n\in\num$,
\begin{enumerate}[(i)]
    \item $\reconstruct^n( z'_0)$ is a $\Phi$-perturbation of the point $z_n:= \reconstruct(z''_n)$.
    \item Since the points $z''_n$ and $z_n$ lie on $X$, there is a sequence of points $\omega_n := h^{-1}(z_n)$ on $\Omega$.
    \item The perturbation $\vec\delta_{n} := \reconstruct^n( z'_0) - z_n$ satisfies $\norm{ \vec\delta_{n} } \leq \left[ 1 + \Csensr(\omega_{n-1}) \right] \norm{ \vec\delta_{n-1} }$.
    \item The perturbation $\vec\epsilon_n := z''_n - z_n$ satisfies  $\norm{\vec\epsilon_n} \leq \Csensr(\omega_{n}) \norm{ \vec\delta_{n} }$.
\end{enumerate}
\end{lemma}
%\eqref{eqn:8iep}
Combining Lemma~\ref{lem:sefjn4} (ii) and (iii) gives
\[%\begin{equation} \label{eqn:knf83}
\norm{ \vec{\epsilon}_n } = d\left( z''_{n}, z_{n} \right) \leq \norm{ \vec\delta_{0} } \Csensr(\omega_{n}) \prod_{i=0}^{n} \left[ 1 + \Csensr(\omega_{i}) \right] .
\]%\end{equation}
Thus the sequence $z_n$ is a pseudo trajectory and similar to Proposition~\ref{prop:srg7} it follows that
\[ \inf_{\bar{w}\in \W} \lambda_1(\bar{w}) - \lambda_1(f,\mu) \leq \int \ln \left[ 1 + \left(1 + \Csens(\omega) \right) \kappa_{\retract} \right] d\mu(\omega) . \]
This completes the proof of  Claim~(iii) and of the theorem. \qed

%-_-_-_-_-_-_-_-_-_-_-_-_-_-_-_-_-_-_-_-_-_-_-_-_-_-_-_-_-_-_-_-_-_-_-_-_-_-_-_-_-_-_-_-_-_-_-_-_-_-_-_-_-_-_-_-_-_-_-_-_-_-_-_-_-_-_-_-_-_-_-_-_-_-_-_-
\subsection{Proof of Lemma~\ref{lem:sdm83}}

Lemma~\ref{lem:sdm83} is where Assumption~\ref{A:w_ext} is needed. We first show how a neighborhood retraction of the attractor leads to an extension $\bar{w}$ of $w$.

\paragraph{$\bar{w}$ from retraction} Since $\retract$ is a retraction, we have $\retract|X \equiv \Id_X$. Let $\proj_2 : \real^{d+L}\to \real^L$ be the projection onto the last $L$ coordinates. Note that $\mathcal{U}_X := \proj_2^{-1}(\mathcal{U})$ is a neighborhood of $X$ in $\real^{d+L}$,. Now define
\[ \bar{w} = (U\phi)\circ \Phi^{-1} \circ \retract : \mathcal{U} \to \ran \phi , \]
and
\[ \alpha := \Phi^{-1} \circ \retract \circ \proj_2 : \mathcal{U}_X \to \Omega . \]
Then we have the following commutations.
\begin{equation} \label{eqn:def:wbar_retrct}
\begin{tikzcd}[column sep = large]
 & & \Omega \arrow{d}[swap]{U\phi} \arrow{ddr}{\Phi} & \\
\mathcal{U}_X \arrow[dashed]{urr}{\alpha} \arrow{r}[swap]{\proj_2} & \mathcal{U} \arrow[red]{drr}{ \red{\retract} } \arrow[dashed, blue]{r}{ \blue{ \bar{w} } } & \ran \phi  & \\
X \arrow{u}{\subset} \arrow{r}[swap]{\proj_2} & \ran \Phi \arrow{ur}[swap]{w} \arrow{u}[swap]{\subset} \arrow{rr}[swap]{\Id} & & \ran \Phi \arrow[bend right = 40]{uul}[swap]{\Phi^{-1}}
\end{tikzcd}
\end{equation}
Note that by definition, $\alpha$ is a continuous map which coincides with $\Phi^{-1} \circ \proj_2$ when restricted to $\ran\Phi$. Moreover, 
\begin{equation} \label{eqn:duu3a}
U\phi\circ \alpha = \bar{w}\circ \proj_2.
\end{equation}

\paragraph{The construction} Set $\omega'_0 = \alpha( z_0')$ and $z''_0 := h(\omega'_0)$ and $\omega_1 := f(\omega'_0)$, as shown below :
\[ \begin{tikzcd}[column sep = large]
\omega_0\ \arrow{d}{h} & & \omega'_0 \arrow{d}{h} \arrow{r}{f} & \omega_1 \arrow{d}{h} \\
z_0 \arrow[]{r}{ +(0, \vec{\delta}_0) } & z_0' \arrow{ur}{\alpha} & z''_0 \arrow{r}{\reconstruct} & z_1
\end{tikzcd} \]

\paragraph{Proof of Claim (i)} We first obtain a bound for $\Phi(\omega'_0) - \Phi\left( \omega_0 \right)$.  Note that
\[\begin{split}
\Phi(\omega'_0) &= \Phi\circ \alpha\left( z_0 + (0, \vec{\delta}_0) \right) = \Phi\circ \Phi^{-1} \circ \retract \circ \proj_2 \left( z_0 + (0, \vec{\delta}_0) \right)   = \retract \circ \proj_2 \left( z_0 + (0, \vec{\delta}_0) \right) \\
&= \retract\left( \Phi(\omega_0) + \vec\delta_0 \right) .
\end{split}\]
Therefore
\begin{equation} \label{eqn:sdkn9}
\norm{ \Phi(\omega'_0) - \Phi\left( \omega_0 \right) } = \norm{ \retract\left( \Phi(\omega_0) + \vec\delta_0 \right) - \retract\left( \Phi(\omega_0) \right) } \leq \kappa_{\retract} \norm{\vec\delta_0} .
\end{equation}
We next estimate the gap $\phi(\omega'_0) - \phi( \omega_0)$. By the definition of the constant $\Csens(\omega_0)$ and by \eqref{eqn:sdkn9},
\begin{equation} \label{eqn:jb37s}
\norm{ \phi(\omega'_0) - \phi(\omega_0) } \leq \Csens(\omega_0) \norm{ \Phi(\omega'_0) - \Phi\left( \omega_0 \right) } \leq \Csens(\omega_0) \kappa_{\retract} \norm{\vec\delta_0}.
\end{equation}
Equip the space $\real^{d+L}$ with the norm $\norm{(x,y)}_{\real^{d+L}} := \norm{x}_{\real^{d}} + \norm{y}_{\real^{L}}$. Then we have 
\[  \begin{split}
    \vec\epsilon_0 &= \norm{ z''_0 - z'_0 } = \norm{ \Matrix{ \begin{array}{c} \phi(\omega'_0) \\ \Phi(\omega'_0) \end{array} } - \Matrix{ \begin{array}{c} \phi(\omega_0)  \\ \Phi(\omega_0) \end{array} } } = \norm{ \phi(\omega'_0) - \phi(\omega_0) } + \norm{ \Phi(\omega'_0) - \Phi(\omega_0) - \delta_0 } \\
    & \leq \left(1 + \Csens(\omega_0) \right) \kappa_{\retract}  \norm{\vec\delta_0}, \quad \mbox{ by \eqref{eqn:sdkn9}, \eqref{eqn:jb37s} } , \\
    &= \Csensr(\omega_0) \norm{ \vec{\delta}_0 } .
\end{split}\]
Similarly, we have
\begin{equation} \label{eqn:s6kmp3}
    \norm{ z''_0 - z'_0 } \leq \norm{ z''_0 - z_0 } + \norm{ z_0 - z'_0 } = \left[ 1 + \Csensr(\omega_0) \right] \norm{\vec\delta_0}.
\end{equation}
This completes the proof of Clam~(i).

\paragraph{Proof of Claim (ii)} Next, by the contractiveness of $g$ from Assumption~\ref{A:g_contract},
\begin{equation} \label{eqn:nmx32}
\norm{ g\left( z''_0 \right) - g\left( z'_0 \right) } \leq \norm{ z''_0 - z'_0} \stackrel{ \mbox{by \eqref{eqn:s6kmp3}} }{\leq} \left[ 1 + \Csensr(\omega_0) \right] \norm{\vec\delta_0} . 
\end{equation}
We have from definition :
\[ \proj_2(z_1) = \proj_2\circ h(\omega_1) = \proj_2\circ h \circ f (\omega'_0) = \proj_2\circ \reconstruct \circ h(\omega'_0) = g(z''_0) . \]
Thus
\begin{equation} \label{eqn:fns3}
\norm{\proj_2(z_1) - \proj_2\circ \reconstruct( z'_0)} = \norm{g(z''_0)) - g(z'_0) } \stackrel{ \mbox{by \eqref{eqn:nmx32}} }{\leq} \left[ 1 + \Csensr(\omega_0) \right] \norm{\vec\delta_0} .
\end{equation}
Finally set $y':= \Phi(\omega_0) + \vec\delta$. Then note that
\[\begin{split}
\proj_1(z_1) &= \proj_1\circ h(\omega_1) = \proj_1\circ h \circ f (\omega'_0) = \proj_1\circ \reconstruct \circ h(\omega'_0) = w\circ \proj_2\circ h(\omega'_0) \\
&= w\circ \Phi(\omega'_0) = w\circ \Phi\circ \Phi^{-1} \circ \retract( y') = \bar{w}(y') \\
& = \proj_1\circ \reconstruct (z'_0).
\end{split}\]
This completes the proof of Claim~(ii) and thus of the lemma. \qed

%-_-_-_-_-_-_-_-_-_-_-_-_-_-_-_-_-_-_-_-_-_-_-_-_-_-_-_-_-_-_-_-_-_-_-_-_-_-_-_-_-_-_-_-_-_-_-_-_-_-_-_-_-_-_-_-_-_-_-_-_-_-_-_-_-_-_-_-_-_-_-_-_-_-_-_-

\subsection{ Proof of Corollary \ref{cor:delay_stab} } \label{sec:proof:delay_stab}

Corollary \ref{cor:delay_stab} satisfies the conditions of Theorem~\ref{thm:lambda1}. The claim will be proved if it can be shown that $\Csens \leq \frac{1}{Q} +Q \bigO{\Delta t}$.
\[ D\Phi(\omega) = \left( D\Phi\left( \Psi^{0\Delta t} \omega \right), D\Phi\left( \Psi^{-\Delta t} \omega \right), \ldots, D\Phi\left( \Psi^{-(Q-1)\Delta t} \omega \right) \right) . \]
Since $\phi$ is a $C^2$ function, 
\[ \norm{ D\phi\left( \Psi^{q\Delta t} \omega \right) - D\Phi\left( \omega \right) } = \bigO{ |q|\Delta t }, \mbox{ as } \Delta t\to 0^+. \]
Thus
\[ \norm{D\Phi(\omega)} = Q\norm{D\phi\left( \omega \right)} + \sum_{q=0}^{Q} \bigO{ q\Delta t } = Q\norm{D\phi\left( \omega \right)} + Q^2\bigO{ \Delta t }, \mbox{ as } \Delta t\to 0^+. \]
Therefore
\[ \Csens(\omega) = \frac{ \norm{ D\phi(\omega) } }{ \norm{ D\Phi(\omega) } } = \frac{ \norm{ D\phi(\omega) } }{ Q\norm{ D\phi(\omega) } + Q^2 \bigO{\Delta t} } = \frac{1}{Q} +Q \bigO{\Delta t}, \; \mbox{ as } \Delta t\to 0^+.
\]
This proves the claim. \qed

%-_-_-_-_-_-_-_-_-_-_-_-_-_-_-_-_-_-_-_-_-_-_-_-_-_-_-_-_-_-_-_-_-_-_-_-_-_-_-_-_-_-_-_-_-_-_-_-_-_-_-_-_-_-_-_-_-_-_-_-_-_-_-_-_-_-_-_-_-_-_-_-_-_-_-_-
\section{Proof of Theorem~\ref{thm:direct}} \label{sec:proof:direct}

By \eqref{eqn:kcnv83}, the direct forecast error can be expressed in terms of the Koopman operator as
\[ \text{error}_{\text{direct}}(n) := \norm{ U^n \phi - \proj_{\mathcal{W}} U^n\phi }_{L^2(\mu)} = \norm{ \left( \Id - \proj_{\mathcal{W}} \right) U^n\phi }_{L^2(\mu)} , \]
as claimed. To proceed further, we have to separately examine the components of $\phi$ along $\Disc$ and its complement. For that purpose, define
\[ \phi^{(d)} := \proj_{\Disc} \phi, \quad \phi^{(c)} := \phi - \phi^{(d)} . \]
This decomposition is possible due to the linearity of $U^n$ and the invariance of the subspaces $\Disc, \Disc^\bot$. Therefore
\begin{equation} \label{eqn:adjp30}
\text{error}_{\text{direct}}(n)^2 = \norm{ \left( \Id - \proj_{\mathcal{W}} \right) U^n\phi }_{L^2(\mu)}^2 = \norm{ \left( \Id - \proj_{\mathcal{W}} \right) U^n\phi^{(c)} }_{L^2(\mu)}^2 + \norm{ \left( \Id - \proj_{\mathcal{W}} \right) U^n\phi^{(d)} }_{L^2(\mu)}^2 .
\end{equation}
We call them the discrete and continuous components respectively, and analyze them separately.

\paragraph{Continuous component}

We begin with a review of some concepts from ergodic theory related to mixing.

\begin{lemma} \label{lem:mixing} [Weak mixing]
Let $(\Omega, \mu, f)$ be a measure preserving system, with the splitting as in \eqref{eqn:def:L2split}. Then for every $\phi_1, \in \Disc^\bot $ and every $\phi_2 \in L^2(\mu)$,
\[\begin{split}
& \lim_{N\to\infty} \frac{1}{N} \sum_{n=0}^{N-1} \abs{ \left\langle \phi_1, U^n \phi_2 \right\rangle_{L^2(\mu)} - \mu\left( \phi_1 \right) \mu\left( \phi_2 \right) } = 0 .\\
& \lim_{N\in \num', N\to\infty} \left\langle \phi_1, U^n \phi_2 \right\rangle_{L^2(\mu)} = \mu\left( \phi_1 \right) \mu\left( \phi_2 \right) ,
\end{split}\] 
where $\num'$ is a subset $\num$ with density $1$.
\end{lemma}

\begin{proof} The first identity follows from \citep[][Mixing Theorem, pg 45]{Halmos1956}. The second identity follows from \cite[][Sec 2.1]{LiaoEtAl2018}.
\end{proof}

If $\num'$ above can be taken to be $\num$ and $\Disc = \{\text{constant}\}$, then the system $(\Omega, f, \mu)$ will be called \emph{strongly mixing}. In other words the following holds.
\begin{equation} \label{eqn:strongmixing}
\lim_{N\to\infty} \left\langle \phi_1, U^N \phi_2 \right\rangle_{L^2(\mu)} = \mu\left( \phi_1 \right) \mu\left( \phi_2 \right) , \quad \forall \phi_1, \phi_2 \in L^2(\mu) .
\end{equation}

We are now ready to prove the following :
\begin{equation} \label{eqn:kj38}
\lim_{n\in \num', n\to\infty} \norm{ \left( \Id - \proj_{\mathcal{W}} \right) U^n\phi^{(c)} }_{L^2(\mu)} = \norm{\phi^{(c)}}_{L^2(\mu)} .
\end{equation}
We will in fact prove the stronger result
\begin{equation} \label{eqn:dp93}
\lim_{n\in \num', n\to\infty} \proj_{\mathcal{W}} U^n\phi^{(c)} = 0.
\end{equation}

\paragraph{Proof of \eqref{eqn:dp93}} By \eqref{eqn:def:W}, $\mathcal{W}$ is spanned by a finite orthonormal basis $\{ \Hbasis_i : i=1,\ldots,M\}$. Then note that
\[ \proj_{\mathcal{W}} \psi = \sum_{i=1}^M \langle \Hbasis_i, \psi \rangle_{L^2(\mu)} \Hbasis_i, \quad \forall \psi\in L^2(\mu) .  \]
Therefore,
\[ \lim_{n\in\num', n\to\infty} \proj_{\mathcal{W}} U^n \phi^{(c)} = \lim_{n\in\num', n\to\infty} \sum_{i=1}^M \langle \Hbasis_i, U^n\phi^{(c)} \rangle_{L^2(\mu)} \Hbasis_i  = \sum_{i=1}^M \lim_{n\in \num', n\to\infty} \langle \Hbasis_i, U^n\phi^{(c)} \rangle_{L^2(\mu)} \Hbasis_i = \sum_{i=1}^M \mu(\Hbasis_i) \mu(\phi^{(c)}) \mathfrak{w_i} = 0 . \]
The identity in \eqref{eqn:dp93} now follows. \qed

\paragraph{Discrete component}  Next, let $z_1, z_2, \ldots$ be an orthonormal basis for $\Disc$, in terms of the Koopman eigenfunctions. Then one has 
\[ \proj_{\Disc} \Hbasis_l = \sum_j a_{l,j} z_j \quad 1\leq l\leq M, \quad a_{l,j} := \langle z_j, \Hbasis_l \rangle_{L^2(\mu)}. \]
Let $\Pi$ be the $\num\times\num$ matrix defined as $\Pi_{j,k} := \left\langle \pi z_j, \pi z_k \right\rangle_{L^2\mu)}$. Then
\[ \sum_{l=1}^M a_{l,k}^* a_{l,j} = \sum_{l=1}^M \langle \Hbasis_l, z_k \rangle_{L^2(\mu)} \langle z_j, \Hbasis_l \rangle_{L^2(\mu)} = \left\langle \sum_{l=1}^M \langle z_k, \Hbasis_l \rangle_{L^2(\mu)} \Hbasis_l, \sum_{l=1}^M \langle z_j, \Hbasis_l \rangle_{L^2(\mu)} \Hbasis_l \right\rangle_{L^2(\mu)} = \left\langle \pi z_j, \pi z_k \right\rangle_{L^2(\mu)} = \Pi_{j,k}. \]
Now let $\phi^{(d)} = \sum_j \phi_j z_j$. Then $U^n \phi^{(d)} = \sum_j \phi_j e^{\iota\omega_j n} z_j$. Therefore,
\[ \left\langle \Hbasis_l, U^n \phi^{(d)} \right\rangle_{L^2(\mu)} = \left\langle \proj_{\Disc} \Hbasis_l, U^n \phi^{(d)} \right\rangle_{L^2(\mu)} = \sum_j \phi_j e^{\iota\omega_j n} \langle \Hbasis_l, z_j \rangle_{L^2(\mu)} = \sum_j \phi_j e^{\iota\omega_j n} a^*_{l,j} . \]
Therefore,
\[ \pi U^n \phi^{(d)} = \sum_{l=1}^{M} \left\langle \Hbasis_l, U^n \phi \right\rangle_{L^2(\mu)} \Hbasis_l = \sum_{l=1}^{M} \sum_j \phi_j e^{\iota\omega_j n} a^*_{l,j} \Hbasis_l . \]
Note that the infinite sequence $\vec{\phi} := \left( \phi_j \right)_{j\in\num}$ is an $\ell^2$ sequence. Define the operator 
\[\Fourier : \ell^2\to \ell^2, \quad \left( \Fourier \vec{\phi} \right)_j := e^{\iota\omega_j} \phi_j, \quad \forall j\in\num.\]
Then %\[ \norm{ U^n \phi }_{L^2(\mu)}^2 = \norm{ \Fourier^n \vec{\phi} }^2 = \norm{ \vec{\phi} }^2 = \norm{\phi}_{L^2(\mu)}^2 , \]
%and
%
\[ \norm{ \pi U^n \phi^{(d)} }_{L^2(\mu)}^2 = \sum_{k,j} \phi_k^* \phi_j e^{\iota\omega_j n} e^{-\iota\omega_k n} \sum_{l=1}^M a_{l,k}^* a_{l,j} = \left( \Fourier^n\vec{\phi} \right)^* \Pi \left( \Fourier^n\vec{\phi} \right) .\]
Therefore
\begin{equation} \label{eqn:lmd93c}
\norm{ \left( \Id - \proj_{\mathcal{W}} \right) U^n\phi^{(d)} }_{L^2(\mu)}^2 = \vec{\phi}^* \Fourier^{n*} \left[ \Id - \Pi \right] \Fourier^n \vec{\phi} .
\end{equation}
The operator $\Fourier$ is a unitary operator which is diagonal with respect to the usual basis of $\ell^2$. Thus
\begin{equation} \label{eqn:dsm40}
\lim_{\Pi \to \Id } \norm{ \left( \Id - \proj_{\mathcal{W}} \right) U^n\phi^{(d)} }_{L^2(\mu)}^2 = 0.
\end{equation}
If the hypothesis space is increased, then $\Pi$ converges strongly to $\Id$ and the above limit is approached. Thus for any $\epsilon>0$, if $\mathcal{W}$ is large enough, then $\norm{ \left( \Id - \proj_{\mathcal{W}} \right) U^n\phi^{(d)} }_{L^2(\mu)}^2 < \epsilon$.

\paragraph{Proof of Theorem~\ref{thm:direct}} Claim~(i) follows from \eqref{eqn:adjp30}, \eqref{eqn:dp93} and \eqref{eqn:dsm40}. In Claim~(ii), if $(f,\mu)$ is weakly mixing, then $\Disc = \{\text{constant}\}$, thus $\phi^{(d)}$ is just the average $\mu(\phi)$. The claim now follows from \eqref{eqn:adjp30} and \eqref{eqn:dp93}. Claim~(iii) follows from the definition of strong mixing and \eqref{eqn:strongmixing}. In Claim~(iv), $\Disc^\bot = \{0\}$, and the claim follows from \eqref{eqn:dsm40}. \qed.

%-_-_-_-_-_-_-_-_-_-_-_-_-_-_-_-_-_-_-_-_-_-_-_-_-_-_-_-_-_-_-_-_-_-_-_-_-_-_-_-_-_-_-_-_-_-_-_-_-_-_-_-_-_-_-_-_-_-_-_-_-_-_-_-_-_-_-_-_-_-_-_-_-_-_-_-
\section{Proof of Theorem~\ref{thm:iterative}} \label{sec:proof:iterative}

We next look at the iterates of the map $\reconstruct$ in \eqref{eqn:def:feedback_1}, with initial conditions in \eqref{eqn:feedback_1_init}. Let $\hat{w} = \hat{w}_1$ as in \eqref{eqn:kcnv83}. The following identity will be used repeatedly.
\begin{equation} \label{eqn:hj3sn3}
\hat{w}\left( U^n \Phi \right) = \hat{w} \circ \Phi \circ f^n \stackrel{ \mbox{ by \eqref{eqn:kcnv83} } }{=} \left( \proj_{\mathcal{W}} U\phi \right) \circ f^n = U^n \proj_{\mathcal{W}} U \phi = U^n \pi U \phi, \quad \forall n\in\num .
\end{equation}
The proof of \eqref{eqn:u_vs_a} will be by induction on $n$. For the base case, note that
\[ z_1 = z_1(\omega_0) = \reconstruct( z_0 ) = \left[ \hat{w}_1\left( \Phi \right), g\circ\left( \phi, \Phi \right) \right] \stackrel{ \eqref{eqn:kcnv83} }{=} \left[ \pi U\phi, U\Phi \right] , \] 
and thus $\Delta u_1 = a_1 = 0^d$ and $\Delta y_1 = b_1 =0^L$. Next suppose that the statement is true up to some $n\in\num$. Using the notation in \eqref{eqn:def:Delta_uy} we have
\[\begin{split}
u_{n+1} &= \hat{w}\left( y_n \right) = \hat{w}\left( U^n\Phi - \Delta y_n \right) = \hat{w}\left(  U^n\Phi \right) - D\hat{w} \vert_{U^n\Phi} \Delta y_n + \bigO{ \norm{\Delta y_n}^2 } \\
&= U^n \pi U \phi - \hat{W}\left( f^n (\cdot) \right) \Delta y_n + \bigO{ \norm{\Delta y_n}^2 }, \quad \mbox{ by \eqref{eqn:def:WG_matrices}, \eqref{eqn:hj3sn3}}.
\end{split}\]
So 
\[\Delta u_{n+1} := U^n \pi U \phi - u_{n+1} = \hat{W}\left( f^n (\cdot) \right) \Delta y_n + \bigO{ \norm{\Delta y_n}^2 } . \]
Similarly,
\[\begin{split}
y_{n+1} &= g\left( u_n, y_n \right) = g\left( U^{n-1} \pi U \phi - \Delta u_n, U^n\Phi - \Delta y_n \right) = g\left( U^n \phi - \Delta u_n - U^{n-1} \Delta U \phi, U^n\Phi - \Delta y_n \right) \\
&= g\left( U^n \phi, U^n\Phi \right) - \nabla_1 g \vert_{ h\circ f^n } \left( \Delta u_n + U^{n-1} \Delta U \phi \right)  - \nabla_2 g \vert_{ h\circ f^n } \Delta y_n + \bigO{ \norm{\Delta u_n}^2 } + \bigO{ \norm{\Delta y_n}^2 } \\
&= U^{n+1} \Phi - G^{(1)} \left( f^n(\cdot) \right) \Delta u_n - G^{(2)} \left( f^n(\cdot) \right) \Delta y_n + c\left( f^n(\cdot)  \right) + \bigO{ \norm{\Delta u_n}^2 } + \bigO{ \norm{\Delta y_n}^2 }.
\end{split}\]
So 
\[\Delta y_{n+1} := U^{n+1} \Phi - y_{n+1} = G^{(1)} \left( f^n(\cdot) \right) \Delta u_n + G^{(2)} \left( f^n(\cdot) \right) \Delta y_n + c\left( f^n(\cdot)  \right) + \bigO{ \norm{\Delta u_n}^2 } + \bigO{ \norm{\Delta y_n}^2 } . \]
Combining we get
\begin{equation} \label{eqn:lpn3}
\Matrix{ \begin{array}{c} u_{n+1} \\ y_{n+1} \end{array} } = 
\hat{M} \left( f^n(\cdot) \right) \Matrix{ \begin{array}{c} u_{n} \\ y_{n} \end{array} }
 + \Matrix{ \begin{array}{c} 1 \\ c\left( f^n(\cdot) \right) \end{array} }
 + \Matrix{ \begin{array}{c} \bigO{ \norm{\Delta y_n}^2 } \\  \bigO{ \norm{\Delta u_n}^2 } + \bigO{ \norm{\Delta y_n}^2 } \end{array} }
\end{equation}
The evolution equation~\ref{eqn:lpn3} for $(u_n, y_n)$ is thus the addition of the Taylor series error terms to the evolution equation \eqref{eqn:ab_exact} for $(a_n, b_n)$. Claim~(i) and \eqref{eqn:u_vs_a} immediately follows. Since the evolution of $(a_n, b_n)$ is that of a perturbed random cocycle, Theorem~\ref{thm:GraphT_growth} applies and Claims~(ii) and (iii) follow. This completes the proof of Theorem~\ref{thm:iterative}. \qed

%-_-_-_-_-_-_-_-_-_-_-_-_-_-_-_-_-_-_-_-_-_-_-_-_-_-_-_-_-_-_-_-_-_-_-_-_-_-_-_-_-_-_-_-_-_-_-_-_-_-_-_-_-_-_-_-_-_-_-_-_-_-_-_-_-_-_-_-_-_-_-_-_-_-_-_-
\bibliographystyle{abbrvnat}
\bibliography{\Path References,\Path temp}
\end{document}